\documentclass[11pt]{amsart}
\usepackage{hyperref}
\usepackage{amsmath, latexsym, amsfonts, amssymb, amsthm, amscd, color}

\DeclareMathOperator{\wlim}{w-\lim}

\DeclareMathOperator{\diag}{diag}

\newcommand{\downto}{\downarrow}

    \def\beq{\begin{eqnarr}}
    \def\eeq{\end{eqnarr}}
    \def\beqq{\begin{eqnarr*}} 
    \def\eeqq{\end{eqnarr*}} 

\newtheorem {definition}{Definition}

\newtheorem{proposition}[definition]{Proposition}
\newtheorem {lemma}[definition]{Lemma}

\newtheorem{remark}[definition]{Remark}

\newtheorem {theorem}[definition]{Theorem}
\newtheorem {rem}[definition]{Remark}

\newcommand{\cQ}{{\mathcal Q}}

\newcommand{\E}{\mathbb E}
\newcommand{\N}{\mathbb{N}}
\renewcommand{\P}{\mathbb{P}}

\newcommand{\R}{\mathbb{R}}
\newcommand{\eps}{\varepsilon}
\newcommand{\1}{\mathbf{1}}

\title{Real Self-Similar Processes Started from the Origin}
\author{Steffen Dereich}
\address{Steffen Dereich, Institut f\"ur Mathematische Statistik, Westf\"ahlische Wilhelms-Universit\"at M\"unster, Orl\'eans-Ring 10, 48149 M\"unster, Germany}

\author{Leif D\"oring}
\address{Leif D\"oring, Departement Mathematik, ETH Z\"urich, R\"amistrasse 10, 8092 Z\"urich, Switzerland}

\author{Andreas E. Kyprianou}
\address{
Andreas E. Kyprianou, 
Department of Mathematical Sciences,
University of Bath,
Claverton Down,
Bath BA2 7AY,
United Kingdom
}

\begin{document}

\maketitle

\begin{abstract}
	Since the seminal work of Lamperti there is a lot of interest in the understanding of the general structure of self-similar Markov processes. Lamperti gave a representation of positive self-similar Markov processes with initial condition strictly larger than $0$ which subsequently was extended to zero initial condition.\\ For real self-similar Markov processes (rssMps) there is a generalization of Lamperti's representation giving a one-to-one correspondence between Markov additive processes and rssMps with initial condition different from the origin.\\	
	 We develop fluctuation theory for Markov additive processes and use Kuznetsov measures to construct the law of transient real self-similar Markov processes issued from the origin. The construction gives a pathwise representation through two-sided Markov additive processes extending the Lamperti-Kiu representation to the origin.
\end{abstract}		
\tableofcontents

\section{Introduction}
	A fundamental property appearing in probabilistic models is self-similarity, also called the scaling property. In the context of stochastic processes, this is the phenomenon of scaling  time and space in a carefully chosen manner such that there is distributional invariance. An example of the latter is Brownian motion for which the distribution of $(cB_{c^{-1/2}t})_{t\geq 0}$ and $(B_{t})_{t\geq 0}$ is the same for any $c>0$; its so-called scaling index is thus understood as $\frac 1 2$. A natural question is if the knowledge of the scaling property alone implies structural properties for a given model and whether such can be used to deduce non-trivial implications.	 In this article we focus on the case of  Markov processes taking values in $\R$ that fulfil the same scaling relation as a Brownian motion, except the scaling index is taken more generally to be $\alpha>0$ rather than  $\frac{1}{2}$. In particular we focus on entrance laws of such processes from the origin, a problem which, although well understood in the case of Brownian motion, is more difficult to address in the the general setting of real self-similar Markov processes.   \smallskip

Before coming to our main results we review results and ideas for self-similar Markov processes with non-negative sample paths. 
	
\subsection{Positive Self-Similar Markov Processes}	
A strong Markov family $\{\P^z,z> 0\}$ with c\`adl\`ag paths on the state space $[0,\infty)$ - $0$ being an absorbing cemetery state -
 is called \emph{positive self-similar Markov process of index $\alpha>0$} (briefly \emph{pssMp}) if the scaling property holds:
 \begin{align}\label{self_sim}
   \text{the law of $(cZ_{c^{-\alpha}t})_{t\geq 0}$ under $\P^z$ is $\P^{cz}$}
 \end{align}
for all $z, c>0$. Here, and in what follows, $Z=(Z_t)_{t\geq 0}$ denotes the canonical process. The analysis of positive self-similar processes is typically based on the Lamperti representation (see for instance Chapter 13 of \cite{kyprianou}). It  ensures the existence of a L\'evy process $(\xi_t)_{t\geq 0}$, possibly killed at an exponential time with cemetery state $-\infty$, such that, under $\P^z$ for $z>0$,
\begin{align*}
	Z_t=\exp(\xi_{\varphi^{-1}(t)}),\qquad t\geq 0,
\end{align*}
where $\varphi(t)=\int_0^t \exp(\alpha \xi_s)ds$ and the L\'evy process $\xi$  is started in $\log(z)$. We use the convention that $\exp(\xi_{\varphi^{-1}(t)})$ is equal to zero, if $t\not\in \varphi([0,\infty))$.\medskip

It is a consequence of the Lamperti representation that pssMps can be split into two regimes: 
\begin{align*}
 \text{\textbf{{\bf (R)}}}&\qquad\,\,\,\P^z(T_{0} <\infty)=1\text{ for all }z>0\qquad\Longleftrightarrow\qquad\xi \text{ drifts to }-\infty\text{ or is killed} \\
\text{\textbf{{\bf (T)}}}&\qquad\,\,\,\P^z(T_{0}<\infty)=0\text{ for all }z>0\qquad \Longleftrightarrow\qquad\xi\text{ drifts to }+\infty \text{ or oscillates}
\end{align*}

Two major questions remained open after Lamperti:\smallskip

(i) How to extend a  pssMp after hitting $0$ in the  recurrent regime {\bf (R)} with an instantaneous entrance from zero?\smallskip

 (ii) How to start a pssMp from the origin in the transient regime {\bf (T)}? More precisely, one asks for extensions $\{\P^z,z\geq 0\}$ with the Feller property so that in particular  $\P^0:=\wlim_{z\downarrow 0}\P^z$ exists in the Skorokhod topology. \smallskip

Both questions have been solved in recent years:  In the recurrent regime it was proved by Fitzsimmons \cite{Fitzsimmons} and Rivero \cite{Rivero2} that there is a unique recurrent self-similar Markov extension (or equivalently a self-similar excursion measure with summable excursion lengths) that leaves zero continuously if and only if 
\begin{align}\label{C1}
	E[e^{\lambda \xi_1}]=1\text{ for some }0<\lambda<\alpha.
\end{align}	
For the transient regime, it was shown in Chaumont et al. \cite{ChaumontEtAl} and also in Bertoin and Savov \cite{BertoinSavov} that, if the ascending ladder height process of $\xi$ is non-lattice, the weak limit $\P^0$ exists if and only if the weak limit of overshoots
\begin{align}\label{C2}
	 O:=\wlim_{x\uparrow \infty} (\xi_{\tau_x}-x)\quad\text{ exists},
\end{align}
where $\tau_x:=\inf\{t:\xi_t\geq x\}$. If \eqref{C2} holds then one says $\xi$ has stationary overshoots.\smallskip

  There are different ways of proving the results that involve more or less complicated constructions for the underlying L\'evy process $\xi$. The construction that appears to be the most natural to us, in the sense that it works for {\bf (R)} and {\bf (T)}, was carried out for the recurrent regime by Fitzsimmons and shall be developed in this article for the transient regime.\smallskip
  
It has been known for a long time in probabilistic potential theory that excessive measures of Markov processes are closely linked to the entrance behaviour from so called entrance boundaries. One way the relation is implemented involves Markov processes with random birth and death (Kuznetsov measures) and apart from diffusion processes not many examples are known in which the general theory yields concrete results. Self-similar Markov processes form a nice class of non-trivial examples for which the abstract theory gives explicit results. The essense is a combination of Lamperti's representation with Kaspi's theorem on time-changing Kuznetsov measures. Excursions away from $0$ of a pssMp are governed by an excursion measure $n$  corresponding to a particular excessive measure for the pssMp that itself turns out to be a transformation of an invariant measure of $\xi$. Invariant measures for L\'evy processes are known explicitly from the Choquet-D\'eny theorem, hence, excursion measures for pssMp can be identified and constructed through Kuznetsov measures.\smallskip

It is interesting to observe that the constructions of self-similar excursion measures $n$ as Kuznetsov measures work in the recurrent and transient  regimes without using Conditions \eqref{C1} and \eqref{C2} [Note that we also interpret $\P^0$ as a normalized ``excursion measure'' even though an excursion starts at $0$, does not return to $0$ and $\P^0$ must be a probability measure]. The necessity and sufficiency enters as follows:
\begin{itemize}
	\item[\textbf{{\bf (R)}}] Condition \eqref{C1} is necessary and sufficient to construct from $n$ a Markov process by gluing excursions drawn according to a point process of excursions (using Blumenthal's theorem on It\=o's synthesis).
	\item[\textbf{{\bf (T)}}] To define $\P^0$ as normalized ``excursion measure'' the Kuznetsov measure needs to be finite and this is equivalent to Condition \eqref{C2}, see Remark \ref{remar} below.
\end{itemize}
We present our constructions for {\bf (T)} directly in the more general setting of real self-similar Markov processes. 
\begin{remark}
The argument of Fitzsimmons \cite{Fiz} for recurrent extensions extends readily to the real-valued setting by replacing L\'evy processes through MAPs. Since our main purpose is to show how the potential theoretic approach has to be carried out in the transient case and since the article is already technical enough we do not address this topic here.
\end{remark}

\subsection{Real Self-Similar Markov Processes - Main Results} Let $\mathbf D^*$ be the space of c\`{a}dl\`{a}g functions $w:\R_+\to\R$ with $0$ as absorbing cemetery state
endowed with the Skorokhod topology and the corresponding Borel $\sigma$-field $\mathcal D^*$. A family of distributions  $\{\P^z\colon z\in \R\backslash \{0\}\}$ on $(\mathbf D^*, \mathcal D^*)$ is called \emph{strong Markov family} on $\R\backslash \{0\}$ if the canonical process $(Z_t)_{t\geq 0}$ is strong Markov with respect to the canonical right continuous filtration. 
If additionally the process satisfies the  scaling property \eqref{self_sim} for all $z\in\R\backslash \{0\}$ and $c>0$, then the process is  called \emph{real self-similar Markov process}. A result of Chaumont et al. \cite{CPR}, completing earlier work of Kiu \cite{Kiu}, is that for any real self-similar Markov process,  there is a Markov additive process $(\xi_t,J_t)_{t\geq 0}$ on $\R\times \{\pm 1\}$ such that under $\P^z$ the canonical process can be represented as
\begin{align}\label{LK}
   Z_t =  \exp\bigl( \xi_{\varphi^{-1}(t)}\bigr) \,J_{\varphi^{-1}(t)},\quad  t\geq 0,
\end{align}
where $\varphi(t)=\int_0^t \exp(\alpha \xi_s)ds$ and $(\xi_0,J_0)=(\log |z|,[z])$ with
$$
[z]=\begin{cases} 1 &:\mbox{ if } z>0 \\-1 &:\mbox{ if }z<0\end{cases}.
$$
Again we use the convention that  $\exp\bigl( \xi_{\varphi^{-1}(t)}\bigr) \,J_{\varphi^{-1}(t)}$ is equal to zero if $t\not\in \varphi([0,\infty))$.
 A \emph{Markov additive process} (MAP) is a stochastic process $(\xi_t,J_t)_{t\geq 0}$ on $\R\times E$, where $E$ is a finite set, if $(J_t)_{t\geq 0}$ is a continuous time Markov chain on $E$ (called modulating chain) and, for any $i\in E$ and $s,t\geq 0$,
\begin{align*}
	&\text{given }\{J_t=i\},\text{ the pair }(\xi_{t+s}-\xi_t,J_{t+s})_{s\geq 0} \text{ is independent of the past}\\
	&\qquad\text{ and has the same distribution as }(\xi_s, J_s)_{s\geq 0}\text{ under }P^{0,i}.
\end{align*}
If the MAP is killed, then $\xi$ shall be set to $-\infty$.
An important feature of MAPs that will be used throughout our analysis is their close proximity to L\'evy processes. For a textbook treatment of standard results for MAPs see for instance Asmussen \cite{Asmussen}.
\begin{proposition}\label{p1}
	A process $(\xi,J)$ is a MAP if and only if there exist sequences of
	\begin{itemize}
		\item  L\'evy processes $(\xi^{n,i})_{n\in\N_0}$, iid for $i\in E$ fixed,
		\item  real random variables $(\Delta^n_{i,j})_{n\in \N}$, iid for $i,j\in E$ fixed,
	\end{itemize}
  independent of $J$ and of each other such that, if $T_n$ is the $n$th jump-time of $J$, then $\xi$ can be written as
	\begin{align*}
		\xi_t=
		\begin{cases}
			\xi_0+\xi_t^{0,J_0}&: t< T_1\\		
		\xi_{T_n-}+\Delta^n_{J_{T_n-},J_{T_n}}+\xi^{n,{J_{T_n}}}_{t-T_n} 
&: t\in [T_n,T_{n+1}), t<\mathbf k\\
\xi_t=-\infty&:t\geq \mathbf k
 \end{cases},
	\end{align*}
where the killing time $\mathbf k$ is the first time one of the appearing L\'evy processes is killed:
\begin{align*}	
	\mathbf{k} =\inf\Big\{t>0\colon  \exists n\in\N_0, T_n\leq t \text{ such that }\xi^{n,J_{T_n}}\text{ is killed at time }{t-T_n} \Big\}.
\end{align*}
\end{proposition}

In words, the idea behind a MAP is as follows: There is a time-dependent random environment governed by the state of $J$ and for every state there is a corresponding L\'evy process $\xi^i$ with triplet $(a_i,\sigma^2_i,\Pi_i)$. If $J$ is in state $i$, then $\xi$ evolves according to a copy of $\xi^i$. Once $J$ changes from $i$ to $j$, which happens at rate $q_{i,j}$,  $\xi$ has an additional transitional jump $\Delta_{i,j}$ and until the next jump of $J$, $\xi$ evolves according to a copy of $\xi^j$. The MAP is killed as soon as one of the L\'evy processes is killed.\medskip

  Consequently, the mechanism behind the Lamperti-Kiu representation is simple: $J$ governs the sign of $Z$ and on intervals with constant sign the Lamperti-Kiu representation simplifies to the Lamperti representation. 
  
 \begin{remark}
 The MAP formalism for the Lamperti-Kiu representation does not appear in \cite{CPR} but has been introduced in \cite{Kyp}.
  \end{remark} 
  
  From now on we assume
  \bigskip
\begin{align*}
	\hspace{-8cm} \textbf{{\bf (I)}}\qquad J\text{ is irreducible on $\{\pm1\}$} 
\end{align*}

\bigskip

 that is, neither $1$ nor $-1$ is absorbing. Assumption {{\bf (I)}} involves no loss of generality: If $J$ is not irreducible, then \eqref{LK} implies that the self-similar process changes sign at most once, thus, can be treated as positive (or negative) self-similar process to which the results for pssMps apply.  Note also that {{\bf (I)}} ensures that the modulating chain $J$ has a unique stationary distribution, which we denote by $\pi=(\pi_+,\pi_-)$. In keeping with this notation, we shall also write the off diagonal elements of the transition matrix of $J$ as $q_{+,-}$ and $q_{-,+}$.
 
 \smallskip

We also assume
 \begin{align*}
	\hspace{-8cm} \textbf{{\bf (NL)}}\qquad
	 \xi \text{ is non-lattice} 
\end{align*}

which is a standard assumption to avoid technicalities. The reader is referred to the discussion at the end of Appendix A.3 for some discussion on this assumption.
    
Throughout the article some notation for first hitting times is used: For a real-valued process
\begin{align*}
	T_{\{0\}}=\inf\{t: Z_t=0\}\qquad\text{and}\qquad
	T_{\eps}=\inf\{t: |Z_t|\geq \eps\}
\end{align*}
and for a bi-variate process $(Z^1,Z^2)$	
\begin{align}\label{lala}
\begin{split}
	\tau^-_x=\inf\{t: Z^1_t\leq x\}\qquad\text{and}\qquad
	\tau^+_x=\inf\{t: Z^1_t\geq x\} 
\end{split}	
	\end{align}	
for $x\in \R$.  

\smallskip
 
Analogously to   L\'evy processes one knows that an unkilled MAP $(\xi,J)$ almost surely either drifts to $+\infty$ (i.e. $\lim_{t\uparrow\infty}\xi_t=+\infty$), drifts to $-\infty$ (i.e. $\lim_{t\uparrow\infty}\xi_t=-\infty$) or oscillates (i.e. $\liminf_{t\uparrow\infty}\xi_t=-\infty$ and $\limsup_{t\uparrow\infty}\xi_t=+\infty$). As for pssMps a simple 0-1 law for real self-similar Markov processes can be deduced from the Lamperti-Kiu representation:
 \begin{proposition}
	If $(\xi,J)$ is the Markov additive process corresponding to a real self-similar Markov process through the Lamperti-Kiu representation, then one has the following dichotomy:
	\begin{align*}
 \text{\rm\textbf{{\bf (R)}}}\qquad\,\,\,\P^z(T_{\{0\}}<\infty)=1\text{ for all }z\neq 0\qquad\Longleftrightarrow&\qquad(\xi,J) \text{ drifts to }-\infty \text{ or is killed,} \\
\text{\rm\textbf{{\bf (T)}}}\qquad\,\,\,\P^z(T_{\{0\}}<\infty)=0\text{ for all }z\neq 0\qquad\Longleftrightarrow&\qquad(\xi,J)\text{ drifts to }+\infty \text{ or oscillates}.
\end{align*}
\end{proposition} 
 The proof is very close in spirit to the proof of the analogous result for pssMps (see for instance Chapter 13 of \cite{kyprianou}).\medskip
 
For the rest of this article we assume {\bf (T)} and ask for the existence and a construction of a measure $\P^0$ on the Skorokhod space $(\mathbf D(\R),\mathcal D(\R))$ of c\`{a}dl\`{a}g functions $w:\R_+\to\R$ such that the extension $\{\P^z:z\in\R\}$ of $\{\P^z:z\in\R\backslash \{0\}\}$ is a self-similar Markov family. In other words, the aim is to extend the Lamperti-Kiu representation to transient self-similar Markov processes that do not have zero as a trap.

 smallskip

  Let $\xi^{+}$ and $\xi^-$ be the L\'evy processes and  $\Delta_{+,-}$ and $\Delta_{-,+}$ the random variables appearing in the representation of $(\xi,J)$ from Proposition~\ref{p1} when applied to the two state MAP of the Lamperti-Kiu representation~(\ref{LK}).
  \bigskip
 \begin{itemize}
\item[{\bf {\bf (C)}}]
$\xi_1$ has finite absolute moment and either of the following holds:
\begin{itemize}
\item[(i)] $(\xi,J)$ drifts to $+\infty$
\item[(ii)] $(\xi,J)$ oscillates and
\begin{align*}
\int_{1}^\infty \frac {x \,\Pi([x,\infty))}{1+\int_{0}^x \int_y^\infty \Pi((-\infty,-z])\, dz\, dy}\, dx<\infty,
\end{align*}
where $\Pi$ is the measure
   $$\Pi=\Pi_+ +\Pi_-+q_{+,-}\mathcal L(\Delta_{+,-})+q_{-,+}\mathcal L(\Delta_{-,+})$$ for the L\'evy measure $\Pi_+$ of $\xi^+$ (resp. $\Pi_-$ of $\xi^-$) and the probability distribution $\mathcal L(\Delta_{+,-})$ of $\Delta_{+,-}$ (resp.  $\mathcal L(\Delta_{-,+})$ of $\Delta_{-,+}$).
   \end{itemize}
\end{itemize}

\bigskip

Condition {\bf (C)} shall be called stationary overshoot condition for $(\xi,J)$ as it is the precise condition for the corresponding MAP to have stationary overshoots in the following sense:
\begin{theorem}\label{over}
If  {\bf (NL)} and {\bf (I)} hold, then
\begin{align*}
	\wlim_{a\to+\infty} P^{0,i}&(\xi_{\tau^+_a}-a \in dx, J^+_{{\tau}^+_a} = j)=	\wlim_{a\to+\infty} P^{-a,i}(\xi_{\tau^+_0} \in dx, J^+_{{\tau}^+_0} = j)
\end{align*}
exists independently of $i\in \{\pm 1\}$ and is non-degenerate if and only if Condition {\bf (C)} holds.
\end{theorem}
Theorem \ref{over} is the MAP version of an important result on the existence of stationary overshoots for L\'evy processes (see for instance Chapter 7 of \cite{kyprianou}) for which  $\Pi$ reduces to the L\'evy measue only. From Theorem \ref{MRT} in the Appendix it follows that stationary overshoots are equivalent to requiring finite mean for the ladder height processes of $(\xi,J)$, the analytic condition is provided in Theorem \ref{theo1509-1}. \medskip

We can now state the main theorem of the present article:

\begin{theorem}\label{thm} 
Suppose $\{\P^z:z\not=0\}$ is a real self-similar Markov process for which the corresponding MAP $(\xi,J)$ satisfies Conditions {\bf (I)} and {\bf (NL)}. Then Condition {\bf (C)} for $(\xi,J)$ is necessary and sufficient 
for the existence of an extension $\{\P^z:z\in\R\}$ on $(\mathbf D(\R),\mathcal D(\R))$ such that the following properties hold:
\begin{enumerate}
\item Under $\P^0$ the process leaves $0$ instantaneously. 
\item The corresponding transition semigroup $(P_t)$ on $\R$ has the Feller property.
\item The family $\{\P^z:z\in\R\}$ is self similar.
\end{enumerate}
Furthermore, $\P^0$ is the unique distribution satisfying one of the properties  (1) or (2).
\end{theorem}

The reader might have realized that Assumption {\bf (I)} excludes the special case of positive self-similar Markov processes that occurs for the trivial case that the Markov chain $J$ is constant and the MAP $(\xi,J)$ reduces to a L\'evy process. In fact, the proof for pssMps is a line-by-line tranlation of the proof given here replacing in all arguments MAPs by L\'evy processes. Since the fluctuation theory for MAPs developed in the Appendix is classical for L\'evy processes, the proof for pssMps only requires the main body of the article which also simplifies drastically in notation.

\subsection{Sketch of the Proof}\label{SectionSketch}
The necessity of Condition {\bf (C)} is straight forward. Combining the Lamperti-Kiu representation and Theorem \ref{theo1509-1} the failure of Condition {\bf (C)} implies
\begin{align}\label{dd}
\begin{split}
	\lim_{|z|\to 0}\P^z(|Z_{T_{\eps}}|<c)
	&=\lim_{|z|\to 0}P^{\log |z|,[z]}\Big(\exp(\xi_{\tau^+_{ \log(\eps)}})<c\Big)\\
	&=\lim_{|z|\to 0}P^{\log |z|,[z]}\Big(\xi_{\tau^+_{ \log(\eps)}}-\log(\eps)<\log (c/\eps)\Big)\\
	&=0
	\end{split}
\end{align}
for any positive $c,\eps$ fixed. Now define 
\begin{align*}
	f(z)=\begin{cases}
		\P^z(|Z_{T_{\eps}}|<c)&:z\neq 0\\
		0&:z=0
		\end{cases},
\end{align*}	
	then, using the calculation from \eqref{dd} and the remark following \eqref{cts} in the Appendix, $f$ is continuous. Hence, for any $\delta>0$ we may choose $a>0$ so that $\sup_{|z|\leq a}f(z)<\delta$. Suppose $\P^0$ is as in Theorem \ref{thm}, then, by the strong Markov property,
	\begin{align*}
		\P^0(|Z_{T_{\eps}}|<c)
		&=\lim_{\eps'\to 0} \int \P^z(|Z_{T_\eps}|<c)\P^0(Z_{T_{\eps'}}\in dz)\\
		&=\lim_{\eps'\to 0}\Big( \int_{|z|\leq a} f(z)\P^0(Z_{T_{\eps'}}\in dz)+ \int_{|z|>a} f(z)\P^0(Z_{T_{\eps'}}\in dz)\Big)\\
		&\leq \delta+ \lim_{\eps'\to 0} \P^0(|Z_{T_{\eps'}}|\in [a,\infty)).
	\end{align*}
	By assumption, under $\P^0$, paths are right-continuous and start from zero so the limiting probability on the right-hand side vanishes. As $\delta$ is arbitrary we proved that $\P^0(|Z_{T_{\eps}}|<c)=0$ for all $\eps, c>0$ which contradicts property (1) of Theorem \ref{thm}.
\smallskip

  The sufficiency of stationary overshoots in Theorem \ref{thm} is non-trivial. Here is the strategy of the proof, the potential theoretic terminology will be clarified in the course of the proof.\medskip

\textbf{Step 1:} Suppose $\{\P^z:z\neq 0\}$ is a Markov family that is continuous in $\R\backslash\{0\}$ with respect to weak convergence on the Skorokhod space - which is true for real self-similar Markov processes due to the Lamperti-Kiu representation - and $\P^0$ is a candidate for the weak limit $\lim_{|z|\downarrow 0}\P^z$. Then a natural guess, for instance from Aldou's criterion, of conditions for the weak convergence is as follows:\smallskip

(a) All overshoots for given levels should converge weakly to the overshoot of $\P^0$ for that level. If so, then nothing has to be controlled past the overshoots due to the strong Markov property and the weak continuity of $z\mapsto \P^z$ away from $0$.\\
(b) The behaviour before the overshoots should be nice in the sense that overshoots over small levels will occur quickly.

To summarize, and this is the content of our Proposition \ref{prop}, to have weak convergence one needs control on overshoots and times of overshoots. For real self-similar Markov processes both quantities can be expressed and analyzed through the Lamperti-Kiu representation and fluctuation theory for Markov additive processes.\smallskip

\textbf{Step 2:} To construct the candidate $\P^0$ assumed in Step 1 we use potential theory: If $\P^0$ is the self-similar process started from zero, then it is a restriction of the Kuznetsov measure $\cQ_\eta$ corresponding to the excessive measure $\eta(dx)=\E^0[\int_0^\infty 1_{(Z_s\in dx)}ds]$, Proposition 3.2 of \cite{FitzMais}. Of course, $\eta$ is not known a priori but this Ansatz leads to a good guess for $\P^0$: $\P^0$ is necessarily the restriction of a Kuznetsov measure for some purely excessive measure.
Since there are many excessive measures of which only one can be the good one, the Ansatz might be too naive.\\ What saves us here is the Lamperti-Kiu representation and Kaspi's time-change theorem: Combined they tell us that the excessive measure should be the Revuz measure of an invariant measure of the MAP. Since invariant measures for MAPs are easy to find, this approach works.\smallskip

Potential theory is most effective when the Markov process is transient. We distinguished two cases in our proof: $(\xi,J)$ drifting to $+\infty$ and $(\xi,J)$ oscillating. In the latter case, the transience is artificially achieved by killing at $T_1$. Such a killing is by no means unnatural: Since only the entrance behavior from $0$ needs clarification, it is equivalent to explain the entrance behavior for the entire process or the process killed at a set bounded away from $0$.

\subsection{Organisation of the Article}	
	The main argument is relatively short but also we need to develop a fair amount of fluctuation theory for Markov additive processes. In order to keep a clear focus the proof is split into two parts: In the next section we give the main argument containing Lamperti-Kiu based calculations for overshoots and times of overshoots (Subsection \ref{S1}) and the potential theoretic construction of $\P^0$ (Subsection \ref{S2}). The fluctuation theory is collected in an Appendix.

\section{Proof}\label{sec:proof}


Throughout the proof, fluctuation theory for Markov additive processes is applied as developed in the Appendix. Unless otherwise stated, we assume throughout that  {\bf (NL)}, {\bf (I)} and {\bf (C)} are in force.
An initial browse of the Appendix at this point may prove to be instructive before digesting the remainder of this section. The main items that are needed from the Appendix is the role of the occupation formula (Theorem \ref{bertoin-spitzer}),  the Markov Renewal Theorem (Theorem \ref{MRT}) and the equivalent conditions for the existence of stationary overshoots (Theorem \ref{theo1509-1}).

\subsection{Convergence Lemma}

The following proposition is the formalization of Step 1 in the sketch of the proof given in Section \ref{SectionSketch}.

\begin{proposition}\label{prop}
	Suppose the following  conditions hold for a strong Markov family $\{\P^z:z\in \R\backslash \{0\}\}$ and a candidate law $\P^0$ on $(\mathbf D(\R), \mathcal D(\R))$:
\begin{enumerate}
\item[(1a)] $\lim_{\eps\to 0} \limsup_{|z| \to 0}\E^z[T_{\eps}]=0$
\item[(1b)] $\wlim_{z\to 0} \P^{z}(Z_{T_{\eps}}\in\cdot)=:\mu_\eps(\cdot)$ exists for all $\eps>0$
\item[(1c)] $\R\backslash\{0\}\ni z \mapsto \P^z$ is continuous in the weak topology on the Skorokhod space
\end{enumerate}
and
\begin{enumerate}
\item[(2a)] $\P^0$-almost surely, $Z_0=0$ and $Z_t\not=0$ for all $t>0$
\item[(2b)] $\P^0((Z_{T_{\eps}+t})_{t\geq 0}\in\cdot)=\P^{\mu_\eps}(\cdot)$ for every $\eps>0$
\end{enumerate}
Then the mapping
$$
	\R\ni z\mapsto \P^z
$$
is continuous in the weak topology on the Skorokhod space.
\end{proposition}
\begin{proof}
	To show convergence in the Skorokhod topology we work with Prokhorov's metric: for $m\in\N$ and two c\`adl\`ag paths $x,y:\R_+\to \R$ define
	\begin{align*}
		d_m(x,y)=\inf\bigl\{\delta>0: \ & \exists \text{ an increasing continuous function }{S}:[0,m]\to[0,\infty) \text{ with }{S}_0=0,\\
		&  \|{S}-\mathrm{id}\|_{[0,m]}\leq \delta\text{ and }\|x\circ {S}-y\|_{[0,m]}\leq \delta\bigr\}
	\end{align*}
	and set
	$$
		d(x,y)=\sum_{m=1}^\infty 2^{-m} (d_m(x,y)+d_m(y,x))\wedge1.
	$$
	Since $d$ generates the Skorokhod topology it suffices to verify that, for arbitrary bounded Lipschitz functions $f:\mathbf D(\R)\to \R$ with Lipschitz constant $\kappa$, say, one has
\begin{align*}
	\E^{z_n}[f(Z)]\to \E^0[f(Z)]
\end{align*}
	for every sequence $(z_n)\to 0$. By property (1b), $\wlim_{z_n\to 0}\P^{z_n}(Z_{T_{\eps}}\in\cdot)= \mu_\eps(\cdot)$, so that by the continuity property (1c) $$\wlim_{z_n\to 0}\int \P^x(\cdot)\, \P^{z_n}(Z_{T_\eps}\in dx)=\int \P^x(\cdot) \,\mu_\eps(dx)= \P^{\mu_\eps}(\cdot).$$
In combination with the Markov property and property (2b) we get
\begin{align*}
	\wlim_{z_n\to 0}\P^{z_n}\big((Z_{T_{\eps}+\cdot})\in\cdot\big)
	=\wlim_{z_n\to 0}\int \P^x(\cdot)\,\P^{z_n}(Z_{T_\eps}\in dx)
	 = \P^{\mu_\eps}(\cdot)=\P^0\big((Z_{T_{\eps}+\cdot})\in\cdot\big).
\end{align*}
	Using the Skorokhod coupling we can define c\`adl\`ag processes $Z^0,Z^1,Z^2,\dots$ on an appropriate probability space $(\Omega,\mathcal F, P)$ on which
	\begin{itemize}
		\item $\mathcal L(Z^n)=\mathcal L_{z_n}(Z)$ for $n\in\N$ and $\mathcal L(Z^0)=\mathcal L_{0}(Z)$,
		\item $(Z^n_{T^n_{\eps}+\cdot}) \to (Z_{T_{\eps}+\cdot})$, almost surely, in the Skorokhod space.
	\end{itemize}
	For $n\in\{0,1,\dots\}$ we denote by $T^n_{\eps}$ the first entrance time of  $Z^n$ into $(-\eps,\eps)^c$. We note that, for every $m\in\N$ and $n,n'\in\{0,1,\dots\}$,
	\begin{align}\label{s}
		d_m(Z^n,Z^{n'})\leq 2\eps+ |T^n_{\eps}-T^{n'}_{\eps}|+ d_m((Z^n_{T^n_{\eps}+\cdot}) , (Z^{n'}_{T^{n'}_{\eps}+\cdot}))
	\end{align}
	which yields
	$$
		d(Z^n,Z^{0})\leq 4\eps+ 2|T^n_{\eps}-T^0_{\eps}|\wedge 1+ d((Z^n_{T^n_{\eps}+\cdot}) , (Z^0_{T^0_{\eps}+\cdot})).
	$$
	Consequently, using Lipschitz continuity of $f$, we get
	$$
		\big|E[f(Z^n)]-E[f(Z)]\big| \leq \kappa \, E[d(Z^n,Z)]\leq 4\kappa \eps+ 2 \kappa\,E[|T^n_{\eps}-T^0_{\eps}|\wedge1\big]+ \kappa\,E\big[d\big((Z^n_{T^n_{\eps}+\cdot}) , (Z^0_{T^0_{\eps}+\cdot})\big)\big].
	$$
	By dominated convergence this gives
	$$
		\limsup_{n\to\infty}\big|E[f(Z^n)]-E[f(Z)]\big| \leq  4\kappa\eps+ 2\kappa\,\limsup_{n\to\infty}E\big[\big|T^n_{\eps}-T^0_{\eps}\big|\wedge1\big]
	$$
	and letting $\eps\to 0$ yields the result since by property (1a),  $\lim_{\eps\to 0}\limsup_{n\to\infty} E[T^n_{\eps}\wedge 1]=0$ and, using (2a), $\lim_{\eps\to 0}E[T^0_{\eps}\wedge 1]=0$.
\end{proof}

\subsection{Verification of Conditions (1a)-(1c)}\label{S1}
	To verify the first three conditions of Proposition \ref{prop} we use the Lamperti-Kiu representation and fluctuation theory for Markov additive processes.

\begin{lemma}\label{pp}
	Condition (1a) from Proposition \ref{prop} holds. 
\end{lemma}
\begin{proof} Using the Lamperti-Kiu representation~(\ref{LK}), one has
	\begin{align*}
	T_{\eps} =\inf\{t\geq 0:|Z_t|\geq \eps\}
	\stackrel{(d)}{=} \inf\left\{t: \xi_{\varphi^{-1}(t)} \geq \log(\eps)\right\}=\varphi(\tau^+_{\log(\eps)})
	\end{align*}
	with $\tau^+_{\log(\eps)}= \inf\left\{t: \xi_t \geq \log(\eps)\right\}.$ Taking expectations and applying the definition of $\varphi$ yields
	\begin{align*}
		\mathbb{E}^z[T_{\eps}] =E^{\log|z|,[z]}[\varphi(\tau^+_{\log(\eps)})]=
		E^{\log|z|,[z]}\Big[  \int^{\tau^+_{\log(\eps)}}_0 e^{\alpha \xi_s}\,ds \Big].
	\end{align*}
	In order to calculate the right-hand side we use the preparations from the Appendix. Let $\hat{P}$ be the law of the 
	dual MAP introduced in Section A.2. It will be useful below to note that, for example,    for bounded measurable functions $f$,
	\[
	E^{z,i}[f(-\xi_t), J_t = j] = \frac{\pi_j}{\pi_i}\hat{E}^{-z, j}[f(\xi_t), J_t = i], \quad z\in\mathbb{R}, i,j\in\{\pm 1\}, t\geq 0.
	\]
	(Compare for instance \eqref{incrementduality} in the Appendix.) Similarly to L\'evy processes, MAPs are spatially homogeneous in the first variable. Using duality in the second and homogeneity in the third equality gives
	\begin{align*}
	\quad  E^{\log|z|,[z]}\Big[\int^{\tau^+_{\log(\eps)}}_0 e^{\alpha \xi_s}ds \Big]
	&= \sum_{j=\pm1}  E^{\log |z|,[z]}\Big[\int^{\tau^+_{\log(\eps)}}_0 e^{\alpha \xi_s}\,ds ; J_{\tau^+_{\log(\eps)}}=j\Big]\\
	 	 	&= 
	\sum_{j=\pm1} \frac{\pi_j}{\pi_{[z]}} \hat{E}^{-\log|z|,j}\Big[\int_0^{\tau^-_{-\log( \eps)}} e^{-\alpha{\xi}_s}\,ds ; J_{\tau^-_{-\log(\eps)}} =[z] \Big]\\
 	 	&= 
	\sum_{j=\pm1}  \frac{\pi_j}{\pi_{[z]}} \hat{E}^{\log(\eps/|z|),j}\Big[\int_0^{\tau^-_0} e^{-\alpha({\xi}_s-\log(\eps))}\,ds ; J_{\tau^-_0} =[z] \Big] \\
	&= \eps^\alpha \sum_{j=\pm1}
	\frac{\pi_j}{\pi_{[z]}} \hat{E}^{\log(\eps/|z|),j}\Big[\int_0^{\tau^-_0} e^{-\alpha{\xi}_s}\,ds; J_{\tau^-_0} =[z]\Big]\\ 
	&\leq \eps^\alpha \sum_{j,k=\pm1}
	\frac{\pi_j}{\pi_{[z]}} \hat{E}^{\log(\eps/|z|),j}\Big[\int_0^{\tau^-_0} e^{-\alpha{\xi}_s}1_{(J_s=k)}\,ds\Big]. 
	\end{align*}
Appealing to Remark \ref{c=1} and Theorem \ref{bertoin-spitzer} in Appendix A.5, we can put the pieces above together and write 	
\begin{align*}
\mathbb{E}^z[T_{\eps}]
\leq \eps^\alpha \sum_{j,k=\pm1}\frac{\pi_j}{\pi_{[z]}}
 \sum_{\ell=\pm1}
 \int_{[0,\infty)}e^{-\alpha y}\,\hat{U}^+_{j,\ell}(dy) \int_{[0,\log(\eps/|z|) ]} e^{-\alpha (\log(\eps/|z|) -z)}\, {U}^+_{k,\ell}(dz),
\end{align*}
where the measure ${U}^+_{k, \ell}$ (resp. $\hat{U}^+_{j,\ell}$) is the potential measure of the ascending (resp. descending) Markov additive ladder height process of ${\xi}$. The reader is referred to Section A.5 of the Appendix for the precise definition. What is important to note for their use in this proof are the following two facts. First, the integrals $\int_{[0,\infty)}e^{-\alpha y}\hat{U}^+_{j,\ell}(dy)$ 
are all finite; see e.g. formula (\ref{LT}) in Section A.5 of the Appendix. 
Second, the Key Renewal-type theorem given in Theorem \ref{MRT} (ii) of Appendix A.6 ensures that $\lim_{|z|\to 0} \int_{[0,\log(\eps/|z|) ]} e^{-\alpha (\log(\eps/|z|) -z)} {U}^+_{k,\ell}(dz) =\pi_{\ell}/\alpha E^{0,\pi}[H^+_1] $ for each $k,\ell\in\{\pm1\}$, where the exact nature of the expectation  $E^{0,\pi}[H^+_1]\in(0,\infty]$ is again explained in the Appendix. All that we need to know at this point of the argument is that it is finite. This follows from Theorem \ref{theo1509-1} in the Appendix thanks to Condition {\bf (C)}.
In conclusion, we have
		\begin{align*}
		\lim_{\eps \to 0}\lim_{|z|\to 0}\E^z[T_{\eps}] \leq \lim_{\eps\to 0}\frac{\eps^\alpha}{\alpha{E}^{0,\pi}[H^+_1]}\sum_{j,k,\ell=\pm1}\frac{\pi_j\pi_\ell}{\pi_{[z]}}
		\int_{[0,\infty)}e^{-\alpha y}\hat{U}^+_{j,\ell}(dy)= 0,
	\end{align*}
and the proof is complete.
	\end{proof}

	In the next lemma we deduce the overshoot distributions for real self-similar Markov processes from the overshoot distributions of the corresponding Markov additive processes. In particular, we prove that Condition~(1b) is satisfied in the setting of Theorem~\ref{thm}.

\begin{lemma}\label{L5}
\begin{itemize}
	\item[(i)] There are proper weak limits
	\begin{align*}
		\wlim_{|z|\to 0}\P^{z}(Z_{T_{\eps}}\in dy)=\mu_\eps(dy),\quad \eps>0,
	\end{align*}	
	if and only if Condition {\bf (C)} holds.
	\item[(ii)] If Condition {\bf (C)} holds, then $\P^{\mu_\eps}(Z_{T_{\eps'}}\in dy)=\mu_{\eps'}(dy)$ for $0<\eps<\eps'$.

\end{itemize}
\end{lemma}

\begin{proof}
	(i) The Lamperti-Kiu representation~(\ref{LK}) and spatial homogeneity of Markov additive processes imply that, for $0<a<b$,
	\begin{align*}
		\P^{z}\big(Z_{T_{\eps}}\in[a,b] \big)&= P^{\log |z|,[z]}\big(\exp\bigl( \xi_{\tau^+_{ \log(\eps)}} \bigr)\in[a,b]; J_{\tau^+_{\log(\eps)}}=1\big)\\
		&= P^{\log |z|,[z]}\big( \xi_{\tau^+_{ \log(\eps)}}-\log(\eps) \in[\log (a/\eps),\log(b/\eps)]; J_{\tau^+_{\log(\eps)}}=1\big)
	\end{align*}
	and, analogously,
$$
\P^{z}\big(Z_{T_{\eps}}\in[-b,-a] \big)= P^{\log |z|,[z]}\big(\xi_{\tau^+_{\log {(\eps)}}}-\log (\eps)\in[\log (a/\eps),\log(b/\eps)]; J_{\tau^+_{\log {(\eps/|z|)}}}=-1\big).
$$
Hence, the distributions $\mathcal L^z(Z_{T_\eps})$ converge for $|z|\to0$ if and only if the overshoots of the Markov additive process converge to a proper limit. This is equivalent to Condition {\bf (C)} by Theorem \ref{over}.\smallskip

%
	(ii) We use the strong Markov property and (i) for an interval $A$:
	\begin{align*}
		 \mu_{\eps'}(A)
		=\lim_{|z|\to 0}\P^{z}\big(Z_{T_{\eps'}}\in A\big)
		=\lim_{|z|\to 0} \int \P^x(Z_{T_{\eps '}}\in A)\P^z(Z_{T_\eps}\in dx)
		=\lim_{|z|\to 0} \int f_A(x)\P^z(Z_{T_\eps}\in dx)
		\end{align*}
with $f_A(x):=\P^x(Z_{T_{\eps '}}\in A)$. Using that $f_A$ is bounded and continuous (see \eqref{cts} and the remark beneath it) and the weak convergence from (i) yields
	\begin{align*}
		\mu_{\eps'}(A)= \int f_A(x)\mu_{\eps}(dx)=\P^{\mu_{\eps}}(Z_{T_{\eps '}}\in A).
	\end{align*}
		
as required.
\end{proof}

A direct consequence of the Lamperti-Kiu representation~(\ref{LK}) is

\begin{lemma}
	Condition (1c) from Proposition \ref{prop} holds.
\end{lemma}

\subsection{Verification of Conditions (2a)-(2b) and Construction of $\P^0$}\label{S2}
In this section we construct the measure $\P^0$ and verify conditions (2a)-(2b) of Proposition~\ref{prop}. Before doing so a brief overview of some notation and results from probabilistic potential theory is given. For a more detailed account the reader is referred to Dellacherie et al. \cite{DMM} (available in French only).\smallskip

\textbf{Notation.} 
We work in the setting of Fitzsimmons and Maisonneuve \cite{FitzsimmonsMaisonneuve} that was also used by Kaspi \cite{Kaspi}. \smallskip

Let $E$ be a locally compact Polish space equipped with its Borel $\sigma$-algebra $\mathcal E$. We extend $E$ by an isolated cemetery state $\partial$ and also equip the extended space $E\cup \{\partial\}$ with  its respective Borel $\sigma$-algebra.  
Let $W$ be the space of functions $w:\R\to E\cup \{\partial\}$ that are $E$-valued and c\`{a}dl\`{a}g on a nonempty interval $(\alpha(w),\beta(w))$ and are equal to $\partial$ on the complement of  $(\alpha(w),\beta(w))$. One calls
$\alpha(w)=\inf\{t:w_t\in E\}$ the time of birth, $\beta(w)=\sup \{t:w_t\in E\}$ the time of death and $\zeta(w):=\beta(w)-\alpha(w)$ the life-time. We denote by $(Y_t(w))_{t\in\R}=(w_t)_{t\in\R}$ the canonical process on $W$ and by $\mathcal G=\sigma(Y_s: s\in \R)$ the canonical $\sigma$-algebra on $W$. We assume that $P=(P_t)_{t\geq 0}$ is the transition semigroup of a Feller process on $E$. A family $(\eta_t)_{t\in\R}$ of measures on $(E, \mathcal E)$ is called an \emph{entrance rule} for $P$ if $\eta_tP_{s-t}\leq \eta_{s}$ for $s>t$, and an \emph{entrance law (at time zero)} if $\eta_t=0$ for $t\leq 0$ and $\eta_tP_{s-t}=\eta_s$ for $s\geq t>0$. In the stationary case where  $\eta_t\equiv m$, $m$ is called  \emph{excessive measure}. Write $\cQ_\eta$ for the Kuznetsov measure corresponding to $(\eta,P)$ and $\cQ_m$ for the stationary case. That is to say, $\cQ_\eta$ is the unique measure on $(W,\mathcal G)$ with one-dimensional marginals $\eta_t$ and transition semigroup $(P_t)$. More precisely
\begin{align*}
	&\quad \cQ_\eta\big(\alpha(Y)<t_1, Y_{t_1}\in d x_1, \cdots ,Y_{t_n}\in dx_n, t_n<\beta(Y)\big)\\
	&=\eta_{t_1}(dx_1)P_{t_2-t_1}(x_1,dx_2)\cdots P_{t_n-t_{n-1}}(x_{n-1},dx_n)
\end{align*}
for $-\infty<t_1<\cdots <t_n<+\infty$. Under a Kuznetsov measure the canonical process is a strong Markov process with random birth and death, i.e. if $\tau$ is a stopping time with respect to the canonical  right continuous filtration $(\mathcal G_t)$ one has 
$$
\cQ_\eta((Y_{\tau +t})_{t\geq 0}\in\,\cdot\,|\mathcal G_\tau)=P^{Y_\tau}(\cdot), \mbox{ \ on } \{\alpha<\tau<\beta\}.
$$
The existence and uniqueness of Kuznetsov measures $\cQ_\eta$ follows from Kuznetsov's work \cite{Kuznetsov}. \smallskip

For the stationary case $\eta_t=m$, a particularly simple construction of Kuznetsov measures was given by Mitro \cite{Mitro} for a Markov process in duality, with respect to $m$, to a second Markov process $(\hat X_t)_{t\geq 0}$ with transition semigroup $(\hat P_t)_{t\geq 0}$, i.e.
\begin{align}\label{dual}
	P_t(x,dy)\,m(dx)=\hat P_t(y,dx)\,m(dy).
\end{align}
In the dual setting $\cQ_m$ 
is the unique measure on $(W, \mathcal G)$ that is translation invariant and  has finite dimensional marginals
\begin{align*}
\cQ(\alpha(Y)<s_l, &Y_{s_l}\in dy_l,\dots, Y_{s_1}\in dy_1, Y_{t_1}\in dx_1,\dots,Y_{t_k}\in dx_k, \beta(Y)>t_k)\\
&=	\int_E m(dx)  \hat P^x[\hat X_{s_1}\in dy_1,\cdots ,\hat X_{s_l}\in dy_l]\, P^x[ X_{t_1}\in dx_1,\cdots  ,X_{t_k}\in dx_k]
\end{align*}
at the times $s_l<\cdots <s_1<0\leq t_1<\cdots <t_k$. 
 In words, to build $\cQ_m|_{\{\alpha<0<\beta\}}$ one samples the invariant measure $m$ at time $0$, and from the outcome starts an independent copy of $X$ to the right and an independent copy of the dual $\hat X$ to the left. An important consequence is that time-reversing the Kuznetsov measure for $(\eta,P)$ yields the Kuznetsov measure for $(\eta, \hat P)$. We should also recall the fact	
\begin{itemize}
	\item $\cQ_m(\alpha=-\infty)=0$ if $m$ is purely excessive (i.e. $m P_t\to 0$ as $t\to\infty$),
	\item $\cQ_m(\alpha>-\infty)=0$ if $m$ is invariant (i.e. $m P_t=m$ for all $t>0$).
\end{itemize}
Later on we will use an entrance law at time zero for the real self-similar Markov process to construct $\cQ_\eta$ - recall that automatically $\alpha=0$ for almost all trajectories - and via $\cQ_\eta$ extend the Markov family $\{\P^z:z\in \R\backslash \{0\}\}$ in the following way:
\begin{lemma}\label{le0109-1}Let $E\cup \{\theta\}$ be a Polish space and let  $\{P^x:x\in E\}$ denote a (killed) Markov family on the space $E$.  Suppose that  $(\eta_t)$ is an entrance law for the Markov family on $E$ for which the corresponding Kuznetsov measure $\cQ_\eta$ fulfills
\begin{itemize}
\item[(i)] $ \cQ_\eta$ is a finite non-trivial  measure
\item[(ii)] $\lim_{t\to 0} Y_t=\theta$, $\cQ_\eta$-a.e., in the space $E\cup\{\theta\}$
\end{itemize}
and define the restriction mapping 
\begin{align*}
\pi: W\to \mathbf D(E\cup \{\theta,\partial\}),\quad \pi(w)_t=
\begin{cases}
	\theta&: t=0\\
	w_t&: t>0
\end{cases}.
\end{align*}
For the normalized measure
\begin{align*}
	 P^\theta(A):= \frac{\cQ_\eta(\pi^{-1}(A) )}{ \cQ_\eta(W)},\quad A\in \mathcal D(E\cup \{\theta\cup \partial\}),
\end{align*}
the extended family $\{P^x:x\in E\cup\{\theta\}\}$ is a (killed) Markov family on  $E\cup\{\theta\}$ so that under $P^\theta$ the canonical  process leaves the initial value $\theta$ instantaneously and satisfies the strong Markov property for  strictly positive stopping times. 
\end{lemma}
The lemma is an immediate consequence of the strong Markov property of Kuznetsov measures.\medskip

In order to construct a good entrance law at zero for the real self-similar Markov process we use the theory of random time-changes for Kuznetsov measures as developped by Kaspi.\\

\textbf{Random Time-Change.}
Let us recall Theorems (2.3) and (2.10) of Kaspi \cite{Kaspi} in the simplest form: Given a (killed) Markov process on $E$ with transition semigroup $(P_t)$ and a locally bounded measurable function $h:E\cup \{\partial\}\to (0,\infty)$ that defines a time-changed Markov transition semigroup via
\begin{align*}
	\tilde P_t f(x):=E^x[f(Z_{S_t})],\quad \text{where }S_t=\inf\Bigl\{s>0: \int_0^s h(Z_u)\,du >t\Bigr\}.
\end{align*}	

Let $\cQ_m$ be the Kuznetsov measure for $(m, P)$ and suppose $B_t:=\int_{(\alpha,t]} h(Y_s)\,ds<\infty$ for almost all realizations (by time homogeneity of $\cQ_m$ it suffices to check the property only for time $t=0$).  Then there is an entrance law $(\eta_t)$ at time zero for $(\tilde P_t)$ such that the corresponding Kuznetsov measure $\tilde \cQ_\eta$ satisfies
\begin{align*}
	\tilde\cQ_\eta(A, \beta>t)=\cQ_m\big(\pi^{-1}(A), 0<B^{-1}_t\leq 1\big),\quad A\in \mathcal G, t>0,
\end{align*}
where
\begin{align*}
    	\pi(Y)_t=
    	\begin{cases}
    		Y_{B^{-1}_t}&:t>0\\
    		\partial&: t\leq 0
\end{cases}.
\end{align*}



In what follows we fix the MAP $(\xi,J)$ on $\R\times \{\pm1\}$ obtained from the given real self-similar Markov process through the Lamperti-Kiu representation and consider the time-change
	\begin{align}\label{bb}
		\tilde P_t f(x,i):=E^{x,i}\big[f(\xi_{S_t}, J_{S_t})\big],\quad \text{where }S_t=\inf\Bigl\{s>0: \int_0^s e^{\alpha \xi_u}\,du >t\Bigr\}.
	\end{align}	
	We use the knowledge of invariant measures for MAPs to construct an entrance law at zero for $(\tilde P_t)$, thus, through concatenation with $h(x,i)=\exp(x)i$, for the real self-similar Markov process. 

\begin{lemma}\label{ll}
	If $(\xi,J)$ drifts to $+\infty$, then there exists a distribution $\P^0$ on $(\mathbf D(\R),\mathcal D(\R))$ for which Conditions (2a) and (2b) of Proposition \ref{prop} hold.
\end{lemma}	
\begin{proof}
We construct an entrance law $(\eta_t)$ at time zero for $(\tilde P_t)$ such that the associated Kuznetsov measure $\tilde\cQ_\eta$ satisfies, for $Y=(Y^1,Y^2)$,
\begin{itemize}
\item[(i)] $\lim_{t\downto0} Y^1_t=-\infty$ and $\beta(Y)=\infty$, $\tilde \cQ_\eta$-a.e.
\item[(ii)] $\tilde \cQ_\eta$ is a finite measure
\item[(iii)] if $\tau_z^+=\inf\{t: Y^1_t\geq z\}$ for $z\in\R$ then
$$\tilde \cQ_\eta \big(\big(Y^1_{\tau_z^+}-z,Y^2_{\tau_z^+}\big)\in (dx,\{i\})\big)= \tilde \cQ_\eta(W)\, \nu(dx,\{i\}),$$ where $\nu$ is the stationary overshoot distribution apparing in Theorem~\ref{MRT} for the MAP $(\xi,J)$.
\end{itemize}
If such a measure $\tilde \cQ_\eta$ can be constructed, then by the Lamperti-Kiu representation~(\ref{LK}) and through Lemma~\ref{le0109-1}, we obtain $\P^0$ from $\tilde \cQ_\eta$ by pathwise applying $h(x,i)=\exp(x)i$ and normalizing to a probabillity measure. The claimed properties (2a) and (2b) follow from the construction.\medskip

 Lemma \ref{doubledual} in the Appendix shows that $(\xi,J)$ and $(\hat \xi,\hat J)$ are in duality on $E=\R\times \{\pm1\}$ with respect to the invariant measure $m(dx,\{i\})=dx\, \pi(i)$.  By assumption $(\xi,J)$ drifts to $+\infty$ and the dual $(\hat \xi, \hat J)$ drifts to $-\infty$. We use Mitro's construction for $\cQ_m$: Sample $(x,i)$ from $m$ and start independently copies of $P^{x,i}$ in the positive time-direction and $\hat P^{x,i}$ in the negative time-direction. We conclude that,  $\cQ_m$-a.e., $\alpha(Y)=-\infty$ and  $\beta(Y)=+\infty$ as well as
 \begin{align}\label{d}
	\lim_{t\to -\infty}Y^1_t=-\infty\quad \text{and}\quad \lim_{t\rightarrow +\infty}Y^1_t=+\infty.
\end{align}			
We now apply Kaspi's time-change as discussed above the lemma to $\cQ_m$ with $B_t=\int_{-\infty}^t   \exp(\alpha Y^1_r)\,dr$. In order to use Kaspi's result we need to check that $B_0<\infty$
	for $\cQ_m$-almost all realizations. From the two-sided construction of $\cQ_m$ it is enough to show that $\int_0^\infty \exp(\alpha \xi_r)\,dr<\infty$ for $\hat P^{x,i}$-almost all $(\xi,J)$. This holds due to the law of large numbers for the dual Markov additive process 
that drifts to $-\infty$. Hence, there is an entrance law $(\eta_t)$ at time zero  for $(\tilde P_t)$	and the corresponding Kuznetsov measure $\tilde \cQ_\eta$ satisfies
	\begin{align}\label{aa}
		\tilde\cQ_\eta(A,\beta>t)=\cQ_m\big(\pi^{-1}(A), 0<B^{-1}_t\leq 1\big),\quad A\in\mathcal F,
	\end{align}
		with $\pi(Y)_t=Y_{B^{-1}_t}$ for ${t>0}$ and $\pi(Y)_t=\partial$ for $t\leq 0$.  
Formula  \eqref{aa} combined with \eqref{d} entail property (i).\medskip
	
	Next we show that the measure $\tilde \cQ_\eta$ is finite. We combine convergence of the overshoots of the MAP with Theorem (2.3) of Kaspi.
	By Theorem~\ref{MRT} in the Appendix, there exists a limiting overshoot distribution for the MAP, say $\nu$. We choose $c>0$ such that $\nu((0,c)\times \{\pm1\})>0$ and set $A=(0,c)\times \{\pm1\}$. Note that the map
$$
\R\ni x\mapsto E^{x,i}\Big[\int_0^\infty \mathbf 1_{A}(\xi_{S_t},J_{S_t)})\, dt\Big]
$$
is lower semi-continuous so that by the  Markov property and weak convergence of the overshoot distribution
	$$
	\liminf_{x\downto-\infty} E^{x,i}\Big[\int_0^\infty \mathbf 1_{A}(\xi_{S(t)},J_{S(t)})\, dt\Big] \geq  E^\nu\Big[\int_0^\infty \mathbf 1_{A}(\xi_{S(t)},J_{S(t)})\, dt\Big]=:\kappa>0.
	$$
	Hence, by Fatou's inequality and the strong Markov property for $\tilde \cQ_\eta$, 
	$$
	\tilde \cQ_\eta\Big(\int_{0}^\infty \textbf 1_A(Y_s)\,ds\Big)\geq \liminf_{\eps\downto 0} \tilde \cQ_\eta\left( E^{Y_\eps}\Big[\int_0^\infty \mathbf 1_{A}(\xi_{S(t)},J_{S(t)})\, dt\Big]\right) \geq \kappa\,\tilde\cQ_\eta(W),
	$$
	where we have used that $\lim_{\eps\downto 0}Y^1_\eps=-\infty$ $\tilde\cQ_\eta$-a.e. 
 	Conversely, Theorem (2.3) of Kaspi relates the occupation time of the set $A$ under the measures $\tilde\cQ_\eta$ and $\cQ_m$ as follows:
	\begin{align*}
	\tilde \cQ_\eta\left(\int_{0}^\infty \textbf 1_A(Y^1_s)\,ds\right)= \cQ_m\left (\int_{[0,1)} \textbf 1_A(Y^1_t)\, e^{\alpha Y_t^1}\, dt\right)= \int_A e^{\alpha y_1} \,m(dy)<\infty.
	\end{align*}
	Here we used in the latter step  that we can interchange the order of integration by Fubini's theorem since  $\cQ_m$ is $\sigma$-finite by construction. Combining the two display formulas gives that  $\tilde \cQ_\eta(W)$ is finite and nonzero. Thus we proved property~(ii).

To prove property (iii) we note that the overshoot distribution is not effected by a time change and hence agrees for $(P_t)$ and $(\tilde P_t)$. Consequently, using the Markov property under the measure $\tilde \cQ_\eta$ we get that
	\begin{align*}
 \tilde \cQ_\eta \big(\big(Y^1_{\tau_z^+}-z,Y^2_{\tau_z^+}\big)\in \cdot\,\big)&=\wlim_{k\downto -\infty} \tilde \cQ_\eta\Bigl[  \tilde P^{Y_{\tau^+_{k}}} \bigl((\xi_{\tau_z^+}-z,J_{\tau_z^+})\in \cdot\,\bigr)\Bigr]\\
&=\wlim_{k\downto -\infty} \tilde \cQ_\eta\Bigl[   P^{Y_{\tau^+_{k}}} \bigl((\xi_{\tau_z^+}-z,J_{\tau_z^+})\in \cdot\,\bigr)\Bigr]\\
	& = \tilde \cQ_\eta(W)\, \nu(\cdot).
	\end{align*}
	 This shows (iii) and the proof is complete.
\end{proof}

The same proof can not be carried out if $(\xi,J)$ oscillates. Chosing the same invariant measure $\eta$ leads to a Kuznetsov measure $\cQ_\eta$ under which trajectories oscillate in both directions of time. Hence, there is no way this construction yields a law $\P^0$ satisfying (2a) of Proposition \ref{prop}. Essentially, the problem is that $Z$ is not transient. To circumvent this issue,  $Z$ is killed at $T_{1}$ and then we proceed similarly as before. This is captured in the lemma below.

\begin{rem}\rm
Before turning to the aforesaid lemma, let us note that the cases that $(\xi,J)$ drifts to $+\infty$ or oscillates can of course be treated both with killing as in the proof of Lemma \ref{lll}. In order to work out clearly the main ideas we prefer to give two proofs. In particular, the reader will find it easier to compare our proof to Fitzsimmons' \cite{Fiz} construction of excursion measures in the recurrent case.
\end{rem}

\begin{lemma}\label{lll}
	If $(\xi,J)$ oscillates, then there exists a distribution $\P^0$ on $(\mathbf D(\R),\mathcal D(\R))$ for which Conditions (2a) and (2b) of Proposition \ref{prop} hold.	
\end{lemma}
\begin{proof}
	We mimik the proof of Lemma \ref{ll} with additional killing.\smallskip

	Recall from Remark \ref{dualcond} in the Appendix that there exists a harmonic function $(x,i)\mapsto U^+_i(x)$  related to the MAP killed  when its first component reaches the positive half-line, henceforth denoted by  $(\xi^\dag,J^\dag)$. The corresponding $h$-transformed process is indicated with the superscript $\downarrow$. We shall also write their respective transition kernels as 
$P^\dag_t((x,i),(dy,\{j\}))$ and
$P^\downarrow_t((x,i),(dy,\{j\}))$, with the addition of a hat to mean the dual map as defined in Section A.2.\\
 Next, we show duality in the sense of \eqref{dual} for $(\hat\xi^\downarrow, \hat J^\downarrow)$ and $(\xi^\dag, J^\dag)$ with respect to the duality measure $m(dx,\{i\})=\pi_i\hat U^+_i(x) dx$ on $(-\infty,0)\times \{\pm 1\}$. The duality comes from the short calculation
	\begin{align*}
		 \hat P^\downarrow_t\big((x,i),(dy,\{j\})\big)m(dx,\{i\})
		&= \frac{\hat U^+_j(y)}{\hat U^+_i(x)}\hat P^\dag_t\big((x,i),(dy,\{j\})\big)  \pi_i\hat U^+_i(x) dx\\
		&=\pi_j \hat U^+_j(y)P^\dag_t\big((y,j),(dx,\{i\})\big) dy \\
		&= P^\dag_t\big((y,j),(dx,\{i\})\big) m(dy,\{j\}),
	\end{align*}
where we used the generic $h$-transform formula for semigroups	\begin{align}\label{ddd}
		P^h_t(x,dy)=\frac{h(y)}{h(x)}P_t(x,dy)
	\end{align}
	for transition probabilities of $h$-transformed processes and the ordinary MAP duality formula
\begin{align*}
\hat P^\dag_t\big((x,i),(dy,\{j\})\big) \pi_idx= P^\dag_t\big((y,j),(dx,\{i\})\big)\pi_jdy
\end{align*}
from Lemma \ref{doubledual} in the Appendix.

\smallskip
	
	Mitro's construction of the Kuznetsov measure $\cQ^\dag_m$ for the killed MAP with respect to $m$ works as follows: Sample $(x,i)\in (-\infty,0)\times \{\pm 1\}$ according to $m$ at time zero and start independently a copy of the killed process $P^{x,i,\dag}$ in positive time-direction and a copy of the conditioned process $\hat P^{x,i,\downarrow}$ in negative time-direction. Since the MAP was assumed to oscillate, the killing time of the former is finite almost surely. Furthermore, the conditioned process drifts to $-\infty$ almost surely by Proposition \ref{cond} in the Appendix. Hence, almost all trajectories $Y=(Y^1,Y^2)$ under $\cQ^\dag_m$ are born at time $\alpha(Y)=-\infty$, die at a finite time $\beta(Y)<+\infty$ and satisfy $\lim_{t\downarrow -\infty}Y^1_t=-\infty$.\smallskip
		
We now apply Kaspi's time-change to $\cQ^\dag_m$ with $B_t=\int_{-\infty}^t   \exp(\alpha Y^1_r)dr$. In order to use Kaspi's result we need to check that $B_0<\infty$ for $\cQ^\dag_m$-almost all realizations. From the two-sided construction of $\cQ^\dag_m$ it is clearly enough to show that $\hat P^{x,i,\downarrow}$-almost surely $\int_0^\infty  \exp(\alpha \xi_r)\,dr<\infty$ for all $(x,i)\in (-\infty,0)\times\{\pm 1\}$. To do so we show finiteness of the expectation:
\begin{align}
		\hat E^{x,i,\downarrow}\Big[\int_0^\infty e^{\alpha \xi_s}\,ds\Big]&=
		\int_0^\infty \hat E^{x,i,\downarrow}\big[e^{\alpha \xi_s}\big]\,ds\notag\\
		&=\int_0^\infty \sum_{j=1,2}\int_\R e^{\alpha y} \hat P^{x,i,\downarrow}(\xi_s\in dy, J_s=j)\,ds\notag\\
		&= \sum_{j=1,2}\int_\R e^{\alpha y}\int_0^\infty  \hat P^{x,i,\downarrow}(\xi_s\in dy,J_s=j)ds\notag\\
		&=:\sum_{j=1,2}\int_\R e^{\alpha y} \hat U^\downarrow\big((x,i),dy,\{j\}\big)\notag\\
		&= \frac{1}{\hat U^+_i(x)}\sum_{j=1,2}\int_0^\infty e^{\alpha y}\hat U^+_j(y)\hat U^\dag\big((x,i),(dy,\{j\})\big)\notag\\
		&\leq \frac{C}{\hat U^+_i(x)}\sum_{j=1,2}\int_0^\infty e^{2 \alpha  y} \hat U^\dag\big((x,i),(dy,\{j\})\big)\notag\\
		&=\frac{C}{\hat U^+_i(x)}\hat E^{x,i,\dag}\Big[\int_0^\infty e^{2 \alpha  \xi_s}\,ds\Big]\notag\\
		&=\frac{C}{\hat U^+_i(x)}\hat E^{x,i}\Big[\int_0^{\tau^+_{0}} e^{2 \alpha  \xi_s}\,ds\Big],
		\label{isfinite}
	\end{align}
	where we used Fubini's theorem and the relation
	\begin{align*}
		\hat U^\downarrow\big((x,i),(dy,\{j\})\big)=\frac{\hat U^+_j(y)}{\hat U^+_i(x)}\hat U^\dag\big((x,i),(dy,\{j\})\big),
	\end{align*}
	with $\hat U^\dag\big((x,i),(dy,\{j\})\big)$ being 
the potential measure of $(\xi^\dagger, J^\dagger)$, (a consequence of \eqref{ddd}) and that the potentials $y\mapsto \hat U^+_j(y)$ grow at most linearly (see Theorem \ref{MRT} of the Appendix).
The right-hand side of (\ref{isfinite}) was already shown to be finite in the proof of Lemma \ref{pp}.\smallskip

Theorems (2.3) and (2.10) of Kaspi \cite{Kaspi}  thus gives us an entrance law $(\eta_t)$ at zero and a corresponding Kuznetsov measure $\tilde \cQ^\dag_\eta$ for the time-changed killed process 
	\begin{align}\label{bbb}
		\tilde P^\dag_t f(x,i):=E^{x,i,\dag}\big[f(\xi_{S(t)}, J_{S(t)})\big], \quad \text{with  }\,\, S_t=\inf\Big\{s>0: \int_0^s \exp(\alpha \xi_u)du >t\Big\},
	\end{align}	
	and furthermore
	\begin{align}\label{a}
		\tilde \cQ^\dag_\eta(A, \beta>t)=\cQ^\dag_m\big(\pi^{-1}(A), 0<B^{-1}_t\leq 1\big),\quad A\in\mathcal F,
	\end{align}
	with $\pi(Y)_t=Y_{B^{-1}_t}$. As in the previous proof, \eqref{a} and the almost sure behavior under $\cQ_m^\dag$ imply the following claim:

	\textbf{Claim:} $\tilde\cQ_\eta^\dag$-almost all trajectories satisfy $\lim_{t\downarrow 0} Y^1_t=-\infty$ and $\beta(Y)<+\infty$.\smallskip

	\textbf{Claim:} $\tilde \cQ_\eta^\dag(W)<\infty$\smallskip
	
	The proof is exactly as in the proof of Lemma \ref{ll}.\smallskip

\textbf{Claim:} $\cQ^\dag_\eta \big(\big(Y^1_{\tau_z^+}-z,Y^2_{\tau_z^+}\big)\in (dx,\{i\})\big)= \cQ_\eta(W)\, \nu(dx,\{i\})$ for all $z<0$.\smallskip
	
	The proof is exactly as in the proof of Lemma \ref{ll} using only $z<0$.\smallskip
	
	Normalizing $\tilde\cQ_\eta^\dag$ to a probability measure and concatenating pathwise with $h(x,i)=\exp(x)i$ yields a law $\P^{0,\dag}$ which is a Kuznetsov measure for the transition semigroup $(P^\dag_t)$ killed at $T_1$. The overshoot distribution under $\P^{0,\dag}$ at levels $\eps<1$ have distributions $\mu_\eps$ (see the proof of Lemma \ref{L5}). Concatenating $\P^{0,\dag}$ with an independent copy of $\P^{\mu_1}$, i.e. running a trajectory under $\P^{0,\dag}$ until $T_1$ and then continuing with an independent copy of $\P^{\mu_1}$, yields $\P^0$. From the above and Lemma \ref{L5}, $\P^0$ has the claimed properties.
\end{proof}

\subsection{Proof of Theorem~\ref{thm}} 
The argument for the necessity of Condition {\bf (C)} was given in Section \ref{SectionSketch}.\medskip

Now suppose {\bf (C)} holds and let $\P^0$ as in Lemma~\ref{ll} or Lemma \ref{lll}, respectively. Then property (1) of Theorem~\ref{thm} is satisfied  and the canonical process under $\P^0$ is strongly Markov for strictly positive stopping times as it is a Kuznetsov measure. In particular, properties (2a) and (2b) of Proposition~\ref{prop} are true. As shown in Section~\ref{S1}, properties (1a) to (1c) are also fulfilled, thus, 
$$
\wlim_{|z|\to0}\P^z=\P^0.
$$
We will use these properties to conclude the remaining assertions of Theorem~\ref{thm}.
 \smallskip

{\bf Step 1:} We show that  the extension $\{\P^z:z\in\R\}$ is Feller. First we show that for arbitrary $t>0$ and continuous and bounded  functions $f:\R\to\R$ the semigroup $P_tf(x)=\E^x[f(X_t)]$ is continuous on $\R$. Suppose that the sequence $(x_n)_{n\in\N}$ converges to $x\in\R$. We know already that $\wlim_{n\to\infty} \P^{x_n}= \P^x$ on the Skorokhod space and it follows that
$$
P_tf(x_n)=\E^{x_n}[f(Z_t)] \to \E^{x}[f(Z_t)]=P_tf(x),
$$
once we ensured that under $\P^x$ the canonical process $Z$ is almost surely continuous in $t$ since point evaluations on the Skorokhod space are continuous on the set of functions being continuous in the respective point. To show this we recall that the paths of real self-similar Markov processes are quasi-left-continuous because the same is true of MAPs, in particular, when they are time changed by the sequence of stopping times that appear in the Lamperti-Kiu transform. In particular, this means that  $Z$ is continuous in $t$, almost surely,  under $\P^x$ if $x\not=0$. In the case where $x=0$ we use the Markov property, to conclude that
$$
\P^0(Z\text{ has jump at }t)=\E^0[\P^{Z_{t/2}}(Z\text{ has jump at }t/2)]=0.
$$
Next, we show that if additionally $f$ vanishes at infinity, then this is also the case for $P_tf$. This is a consequence of the fact that for every $C>0$
$$
\lim_{|x|\to\infty} \P^x(\min_{s\in[0,t]} |Z_s|<C)=0
$$
which itself follows easily from the Lamperti-Kiu representation. Indeed, this estimate  implies that
\begin{align}\label{eq1310-1}
|P_t f(x)|\leq \max_{y:|y|\geq C} |f(y)| + \P^x(\min_{s\in[0,t]} |Z_s|<C)\, \max_{y\in\R} |f(y)|\to \max_{y:|y|\geq C} |f(y)|
\end{align}
for $|x|\to\infty$. Thus, $P_tf$ is vanishing at infinity since $C>0$ is arbitrary. \smallskip

It remains to show the strong continuity for a continuous function $f:\R\to\R$ vanishing at infinity. Let $(t_n)$ be a decreasing sequence with $t_n\to0$ and $(x_n)$ a sequence in $\R$ with either $|x_n|\to\infty$ or $x_n\to x$ for an $x\in\R$. In the case where $|x_n|\to\infty$, with the same estimate as in~(\ref{eq1310-1}), we find
$$
|P_{t_n} f(x_n)-f(x_n)|\leq |P_{t_n}f(x_n)|+|f(x_n)| \to 0.
$$
Moreover, if $x_n\to x$, we get that
$$
P_{t_n} f(x_n)=\E^{x_n}[f(Z_{t_n})]\to \E^x[f(Z_0)]=f(x)
$$
since the functional
$$
\mathbf D(\R)\times [0,\infty)\ni (w,t)\mapsto w_t\in\R
$$
is continuous in $\P^x\otimes \delta_0$-almost all entries. Consequently, one has
$$
\lim_{t\downarrow0} \sup_{x\in\R} |P_tf(x)-f(x)|=0,
$$
since we could otherwise construct sequences $(t_n)$ and $(x_n)$ as above contradicting the above properties (based on the compactness of the one point compactification of $\R$).\smallskip

{\bf Step 2:} Next we show that $\P^0$ is self similar.
For a continuous and bounded functional $f:\mathbf D(\R)\to \R$ we have
$$
\E^0[f(cZ_{c^{-\alpha}\cdot})]=\lim_{z\to0} \E^z[f(cZ_{c^{-\alpha}\cdot})]=\lim_{z\to0} \E^{cz}[f(Z)]= \E^{0}[f(Z)].
$$

{\bf Step 3:} Finally, we show that $\P^0$ is the unique Markovian extension satisfying one of the properties (1) or (2).  Suppose there exists another  Markovian extension satisfying property (1) in the statement of the theorem and denote it by   $\bar \P^0$. Then, for $t>0$,
$$
\bar\P^0 (Z_t\in \,\cdot\,)= \wlim_{\eps\downto0} \bar \P^0(Z_{t+\eps}\in \,\cdot \,) = \wlim_{\eps\downto0} \bar \P^0(\P^{Z_\eps}(Z_{t}\in \,\cdot \,)) = \P^0(Z_t\in\,\cdot\,),
$$
where we used in the first step that  $(Z_t)$ is right-continuous, in the second step the Markov property of $\bar \P^0$ and in the third step  that $Z_{\eps}\Rightarrow \delta_0$ under $\bar \P^0$ and $\wlim_{z\to0} \P^z(Z_t\in\,\cdot\,)= \P^0(Z_t\in\,\cdot\,)$ by the Feller property for $\P^0$.
By using the Markov property one easily sees that the distributions $\bar \P^0$ and $\P^0$ coincide.  

Suppose now that, instead, that $\bar \P^0$ satisfies the Feller property (2) instead of (1). Then using the Feller property twice we get
$$
\bar\P^0 (Z_t\in \,\cdot\,)= \wlim_{x\to 0} \P^x(Z_{t}\in \,\cdot \,) =  \P^0(Z_t\in\,\cdot\,)
$$ 
so that $\bar \P^0$ and $\P^0$ coincide again by the Markov property.

\subsection{Remarks on the Proof}

\begin{rem}\rm
	The way the limiting law $\P^0$ is constructed one can say that the Lamperti-Kiu representation extends in a slightly unhandy way to initial condition $0$. Due to the explicit construction of the Kuznetsov measure from two-sided MAPs one can for instance deduce from almost sure results for MAPs almost sure results for self-similar Markov processes started from zero.
\end{rem}

\begin{rem}[Proof of Theorem \ref{thm} fails if {\bf (C)} fails]\label{remar}\rm
Calculations similar to those from Lemma \ref{ll} (resp.  Lemma \ref{lll}) can be used in order to show that the divergence of overshoots implies $\tilde \cQ_\eta(W)=\infty$ (resp. $\tilde \cQ_\eta^\dag(W)=\infty$). Hence, if Condition {\bf (C)} fails, then necessarily $\tilde\cQ_\eta$ (resp. $\tilde \cQ_\eta^\dag$) is an infinite measure and as such cannot be normalized to a probability measure $\P^0$.
\end{rem}

\begin{rem}\rm\label{ee}
	The previous remark has an interesting consequence: in contrast to other known constructions of $\P^0$ in the setting of pssMps, our construction works irrespectively of Condition {\bf (C)}. When {\bf (C)} fails, then the infinite Kuznetsov measure can still be used to study conditional limits, such as $\lim_{|x|\to 0} \P^z(\cdot\,|\,\text{the interval }[a,b]\text{ is hit})$.
	
\end{rem}

\begin{rem}[Relation to Bertoin, Savov \cite{BertoinSavov}]\rm
	For pssMps Bertoin and Savov constructed $\P^0$  by hand without appealing to the probabilistic potential theory centred around Kuznetsov's measure. Their construction is in the spirit of the Fitzsimmons and Taksar \cite{F} construction of stationary regenerative sets as range of stationary subordinators. In essence, we first constructed a Kuznetsov measure and then produced the so-called quasi-process by taking Palm measures in \eqref{aa} (resp. in \eqref{a}). Bertoin and Savov directly wrote down the quasi-process and their construction only works under Condition {\bf (C)}.
	\end{rem}
	\begin{rem}\rm
The advantage of going the detour through Kuznetsov measures is mostly of technical nature.
It allowed us to write down, with a minimal use of fluctuation theory, the limiting object $\P^0$. For instance, there was no need to use the non-trivial existence of $\hat P^\downarrow$ issued from the origin. Since fluctuation theory is delicate a proof with minimal use is desirable, in particular, for possible futur generalizations to more general domains. One direction for which our construction works but fluctuation theory is not available are multi-self-similar Markov processes introduced in Yor, Jacobson \cite{YJ}.

\end{rem}

\begin{rem}\rm
	For real self-similar Markov processes with jumps only towards the origin a construction of $\P_0$ was already given in \cite{Leif} through jump-type stochastic differential equations. That approach lacks the full generality since the weak uniqueness argument does not extend. It might be an interesting question to ask if the potential theory of the present article can be used to prove the weak uniqueness of the differential equations.
\end{rem}

\appendix

\section{Results for Markov additive processes}

Unlike the case of L\'evy processes, general fluctuation theory for Markov additive processes (MAPs) appears to be relatively incomplete in the literature. Accordingly, in this Appendix, we address those parts of the fluctuation theory that are needed in the main body of the text above. 
\smallskip

The contents of the Appendix is as follows:
\begin{itemize}
	\item[A.1] Basics 
	\item[A.2] Duality
	\item[A.3] Local time and Cox process of excursions
	\item[A.4] Splitting at the maximum
	\item[A.5] Occupation formula
	\item[A.6] Markov Renewal theory
	\item[A.7] Harmonic functions
	\item[A.8] Conditioning to stay positive
	\item[A.9] Laws of large numbers
	\item[A.10] Tightness of the overshoots
\end{itemize}

Unfortunately a complete treatment would require a whole book's worth of text. Therefore, as a compromise and with an apology to the reader, the presentation of A.1 to A.6 mostly highlights selected results and the main steps to prove them. Almost all fluctuation theory can be constructed by analogy with fluctuation theory of L\'evy processes. The selected computations  we dwell on below  pertain largely to the peculiarities that are specific to the case of MAPs. 
Results in A.9 and A.10 are not in analogy to L\'evy processes and non-trivial so full proofs are given.
\bigskip

A.1. {\bf Basics.} Recall that $(\xi_t,J_t)_{t\geq 0}$ denotes a MAP  on $\mathbb{R}\times E$, where $E$ is a finite set.
 Recall also that its  natural filtration is denoted by  $(\mathcal{F}_t)_{t\geq0}$ and its probabilities by  $(P^{x,i})_{x\in\mathbb{R}, i\in E}$. We shall also assume that $E$ is irreducible and aperiodic and hence ergodic. Denote the intensity matrix of $J$ by $Q = (q_{i,j})_{i,j\in E}$. Its stationary distribution is denoted by $\pi = (\pi_1, \cdots, \pi_{|E|})$.

\bigskip

{\it  Unless otherwise stated, we assume throughout  that $\xi$ is non-lattice, that is {\bf (NL)} is in force. }

 \bigskip

Referring to Proposition \ref{p1}, the characteristic exponents of the `pure-state' L\'evy processes appearing in Proposition \ref{p1} will be denoted by $\psi_i(z)=\log \E[\exp(z \xi^i_1)]$, $z\in \mathbb C$, whenever the  right-hand side exists. It suffices for us to deal with the case that $\psi_i(0)=0$ for all $i\in E$, i.e. none of the L\'evy processes are killed. Furthermore, whenever it exists,  define the matrix  $G(z)=\big(G_{i,j}(z)\big)_{i,j\in E} $, where $G_{i,j}(z) = \E[\exp(z \Delta_{i,j})]$, $i,j\in E$. For each $i,j\in E$ such that $i\neq j$, the random variables $\Delta_{i,j}$ have law $F_{i,j}$ corresponding to the distribution of the additional jump that is inserted into the path of the MAP when $J$ undergoes a transition from $i$ to $j$. For convenience we assume that $\Delta_{i,j}=0$ whenever $q_{i,j}=0$ and also set $\Delta_{i,i}=0$ for each $i\in E$. According to Proposition \ref{p1} this assumption is without loss of generality since those transitional jumps never occur.\\

A crucial role will be played by the matrices 
\begin{align}
	F(z):=\text{diag}\big(\psi_1(z),...,\psi_{|E|}(z)\big)+(q_{i,j} G_{i,j}(z))_{i,j\in E},
	\label{piversion}
\end{align}
which are defined on $\mathbb C$ whenever the right-hand side exists. The matrix $F$ is called the matrix exponent of the MAP $(\xi,J)$ because
\begin{align*}
	{E}^{0,i}\big[e^{z \xi_t},J_t=j\big]=\big(e^{F(z)t}\big)_{i,j},\quad i,j\in E,
\end{align*}
for all $z\in \mathbb C$ for which one of the sides is defined.

\bigskip

A.2 {\bf Duality.}
Given the MAP $\xi$ with probabilities $P^{x,i}$, $x\in\mathbb{R}$, $i\in E$, we can introduce the dual process; that is, the  MAP with probabilities $\hat{P}^{x,i}$, $x\in\mathbb{R}$, $i\in E$, whose matrix exponent, when it is defined, is given by, 
\begin{align*}
	\hat{E}^{0,i}\big[e^{z {\xi}_t},J_t=j\big]=\big(e^{\hat{F}(z)t}\big)_{i,j},\quad i,j\in E,
\end{align*}
where 
\[
	\hat{F}(z):=\text{diag}\big(\psi_1(-z),...,\psi_{|E|}(-z)\big)+\hat{Q} \circ G(-z)^{\rm T}
\]
and $\hat Q$ is the intensity matrix of the modulating Markov chain on $E$ with entries given by
\[
\hat{q}_{i,j} = \frac{\pi_j}{\pi_i}q_{j,i}, \qquad i,j\in E.
\]
Note that the latter can also be written $\hat{Q} = \Delta_\pi^{-1} Q^{\rm T} \Delta_\pi$, where $\Delta_\pi = \text{diag}(\pi_1, \cdots, \pi_{|E|})$ and hence, when it exists,
\[
\hat{F}(z) = \Delta_\pi^{-1}F(-z)^{\rm T}\Delta_\pi, 
\]
showing that 
\begin{align}\label{incrementduality}
	\pi_i\hat{E}^{0,i}\big[e^{z {\xi}_t},J_t=j\big]=\pi_j{E}^{0,j}\big[e^{-z {\xi}_t},J_t=i\big].
\end{align}
At the level of processes, one can understand (\ref{incrementduality}) as changing time-directions:

\begin{lemma}\label{duality}We have that  $\{ (\xi_{(t-s)-} -\xi_t, J_{(t-s)-}): s\leq t\}$ under $P^{0,\pi} =\sum_{i=1}^{|E|}\pi_i P^{0,i}$ is equal in law to $\{(\xi_s, J_s) : s\leq t\}$ under $\hat{P}^{0,\pi}$. 
\end{lemma}

Additionally to the ordinary duality \eqref{incrementduality} we will use duality in the general sense of \eqref{dual} for the killed MAP
\[
P^\dag_t((x,i),(dy,\{j\})) = P^{x,i}[\xi_t \in dy, \, \bar\xi_t\leq 0; J_t = j], \qquad x,y\leq 0, t\geq 0, i,j\in E,
\]
where $\bar\xi_t = \sup_{s\leq t}\xi$. The next two duality formulas are called switching identities:

\begin{lemma}\label{doubledual} If $x,y\in\mathbb{R}$ and $i,j\in E$, then
\[
\hat{P}^{x,i}(\xi_t \in dy; J_t = j)\pi_i dx = {P}^{y,j}(\xi_t \in dx; J_t = i)\pi_j dy 
\]
and, for $x,y\leq 0$,
\[
\hat P^\dag_t\big((x,i),(dy,\{j\})\big) \pi_idx= P^\dag_t\big((y,j),(dx,\{i\})\big)\pi_jdy.
\]
\end{lemma}
The proofs of the previous two lemmas are standard, especially in light of the straightforward nature of the analogous proofs for for L\'evy processes (see for example Chapter II of \cite{Bertoin}), and we leave them to the reader.
\bigskip

A.3. {\bf Local time and Cox process of excursions.}  Let  $Y^{(x)}_t = (x\vee\bar\xi_t) - \xi_t$, $t\geq 0$, where we recall that  $\bar\xi_t = \sup_{s\leq t}\xi_s$. 
Following ideas that are well known from the theory of L\'evy processes, it is straightforward to show that, as a pair, the process $(Y^{(x)},J)$ is a strong Markov process.
For convenience, write $Y$ in place of $Y^{(0)}$.
Since $(Y,J)$ is a strong Markov process, by the general theory (c.f. Chapter   IV of \cite{Bertoin}) there exists a local time at the point $(0,i)$, which we henceforth denote by $\{\bar{L}^{(i)}_t: t\geq 0\}$. Now consider the process
\[
\bar{L}_t := \sum_{i\in E} \bar{L}^{(i)}_t, \qquad t\geq 0.
\]
Since, almost surely, for each $i\neq j$ in $E$, the points of increase of  $\bar{L}^{(i)}$ and $\bar{L}^{(j)}$ are disjoint, it follows that  $(\bar{L}^{-1}, H^+, J^+): = \{(\bar{L}^{-1}_t, H^+_t, J^+_t):t\geq 0\}$ is a (possibly killed)  Markov additive bivariate subordinator, where	
\[
H^+_t : = \xi_{\bar{L}^{-1}_t}\text{ and } J^+_t : = J_{\bar{L}^{-1}_t}, \qquad \text{ if } \bar{L}^{-1}_t<\infty,
\]
and $H^+_t : = \infty$ and $J^+_t : = \infty$ otherwise.
Note that the rate at which the process $(\bar{L}^{-1}, H^+, J^+)$ is killed depends on the state of the chain $J^+$ when killing occurs. This will be addressed in more detail shortly. We also note that $\{\epsilon_t: t\geq 0\}$ is a (killed) Cox process, where 
\[
\epsilon_t  = \{\xi_{\bar{L}^{-1}_{t-} + s} - \xi_{\bar{L}^{-1}_{t-}} : s\leq \Delta \bar{L}^{-1}_t\}, \qquad \text{ if }\Delta \bar{L}^{-1}_{t}>0,
\]
and $\epsilon_t = \partial$, some isolated state, otherwise.
Henceforth, write $n_i$ for the intensity measure of this Cox process when the underlying modulating chain $J^+$ is in state $i\in E$.
As a bivariate Markov additive subordinator, the process $(\bar{L}^{-1}, H^+, J^+)$ has a matrix Laplace exponent given by
\[
{E}^{0,i}\big[e^{-\alpha \bar{L}^{-1}_t - \beta H^+_t }, J^+_t = j\big] = \big(e^{- \kappa^+(\alpha, \beta)t}\big)_{i,j},\qquad \alpha,\beta\geq 0,
\]
where the matrix $\kappa^+(\alpha, \beta)$ has the structure
\[
\kappa^+(\alpha,\beta) = \text{diag}\big(\Phi^+_1(\alpha,\beta), \cdots,  \Phi^+_k(\alpha,\beta)\big)  - Q^+\circ G^+(\beta),\qquad \alpha,\beta\geq 0
\]
such that, for $i\in E$,  $\Phi^+_i(\alpha, \beta)$ is the subordinator exponent that describes the movement of  $(\bar{L}^{-1}, H^+)$ when the modulating chain $J^+$ is in state $i$. Moreover, $Q^+$ is the intensity of $J^+$ and the matrix $G^+(\beta) = (G^+(\beta))_{i,j}$ is such that, for $i\neq j$ in $E$, its $(i,j)$-th entry is the Laplace transform of $F^+_{i,j}$, the distribution of the additional jump incurred by  $H$ when the modulating chain changes state from $i$ to $j$. The diagonal elements of $G^+(\beta)$ are set to unity. Note that there is no additional jump incurred by $\bar{L}^{-1}$ when the modulating chain changes state. For future reference, write 
\[
\Phi^+_i(\alpha, \beta) = {n}_i(\zeta = \infty) + a_i \alpha + b_i \beta + \int_0^\infty\int_0^\infty (1-e^{-\alpha x -\beta y})n_i(\zeta \in dx, \epsilon_\zeta \in dy, J_\zeta = i),\qquad \alpha,\beta\geq 0, 
\]
where $a_i, b_i\geq 0$ and $\zeta = \inf\{s\geq 0: \epsilon >0\}$ the excursion length.
Note in particular that the matrix
\[
\kappa^+(0,0) = \diag\big({n}_1(\zeta = \infty), \cdots, {n}_k(\zeta = \infty)\big),
\]
encodes the respective killing rates of $(\bar{L}^{-1}, H^+, J^+)$ when $J^+$ is in each state of $E$.

\bigskip

The assumption that $\xi$ is non-lattice implies that 
 the jump measures associated to $H^+$, namely $n_i(\epsilon_\zeta\in dx, J^+_\zeta = i)$, $i\in E$, and $F^+_{i,j}$, $i\neq j$, $i,j\in E$, are diffuse on $(0,\infty)$. For the sake of brevity, we give no proof of this fact here. Instead we refer to proof of the analogous result for the case of L\'evy processes. In that case,  one may draw the desired conclusion out of, for example,  Vigon's identity for the jump measure of the ascending ladder height process; see Theorem 7.8 in \cite{kyprianou}. As one sees from the proof there, this identity is derived  using the the so-called quintuple law of the first passage problem, which itself follows from a straightforward application of the compensation formula for the Poisson point process of jumps. A quintuple law can also be derived in the MAP setting using the same technique as in the L\'evy setting, where one appeals to an analogue of the compensation formula for the Cox process of jumps. This would also form the basis of the proof that the jump measures associated to $H^+$  are diffuse in the MAP case.

\bigskip

A.4. {\bf Splitting at the maximum.} Now suppose that $\mathbf{e}_q$ is an exponentially distributed random variable with rate $q>0$. Consider a marked version of the Cox process described above in which each excursion $\epsilon_t\neq \partial$ is marked with an independent copy of $\mathbf{e}_q$, denoted by $\mathbf{e}^{(t)}_q$, for $t\geq 0$. Let $\overline{m}_{t} = \sup\{s\leq t: \overline{\xi}_t = \xi_s \} $. 
Poisson thinning dictates that 
$
(\overline{\xi}_{\mathbf{e}_q}, \overline{m}_{\mathbf{e}_q}) 
$ is equal in law to the process $(\bar{L}^{-1}, H^+)$ conditioned on $\{\Delta \bar{L}^{-1}_t  < \mathbf{e}^{(t)}_q \text{ for all }t\geq 0\}$ and stopped with rate matrix 
\begin{align*}
&\quad{\text{diag}(a_1q + {n}_1(\zeta > \mathbf{e}_q), \cdots, a_{|E|} q + {n}_{|E|}(\zeta > \mathbf{e}_q))}&\\
&=\text{diag}(a_1 q+ {n}_1(1- e^{-q\zeta}), \cdots, a_{|E|} q+ {n}_{|E|}(1-e^{-q\zeta}))\\
&=\text{diag}(\Phi^+_1(q,0), \cdots, \Phi^+_{|E|}(q,0)).
\end{align*}
In particular, the conditioned process is stopped at a random time $\theta_q$ with the property that
\[
{P}^{0,i}\big(\theta_q >t \,|\,\sigma\{J^+_s: s\leq t\}\big) = \exp\Big(-\int_0^t \Phi^+_{J_s}(q,0)\,ds\Big).
\]

The aforementioned conditioned process has matrix exponent which can be derived from the matrix exponent $\kappa^+(\alpha,\beta)$. Indeed, whereas in $\kappa^+(\alpha, \beta)$ the pure states are represented as
$
\Phi^+_i(\alpha, \beta) $
in the conditioned process, this is replaced by 
\[
n_i(\zeta = \infty) + a_i \alpha + b_i \beta + \int_0^\infty\int_0^\infty (1-e^{-\alpha x -\beta y})e^{-q x}n_i(\zeta \in dx, \epsilon_\zeta \in dy, J_\zeta = i),\qquad \alpha,\beta\geq 0, 
\]
which is also equal to $ \Phi^+_i(q+\alpha, \beta)-\Phi^+_i(q,0)$. Hence the conditioned process has matrix exponent given by 
\begin{equation}
\tilde{\kappa}^+(\alpha, \beta) := \text{diag}(\Phi^+_1(q+\alpha,\beta) - \Phi^+_1(q,0), \cdots,  \Phi^+_{|E|}(q+\alpha,\beta) -\Phi^+_{|E|}(q,0)) - Q^+\circ G^+(\beta), 
\label{conditionedLT}
\end{equation}
for $\alpha,\beta\geq 0$.

For convenience,  denote by $(\mathcal{L}^{-1}, \mathcal{H}, J^+)$ the process corresponding to  $(\bar{L}^{-1}, H^+)$ conditioned on $\{\Delta \bar{L}^{-1}_t  < \mathbf{e}^{(t)}_q \text{ for all }t\geq 0\}$, i.e. the Markov additive process with joint Laplace exponent give by (\ref{conditionedLT}).
It now follows that the pair $(\overline\xi_{\mathbf{e}_q}, \overline{m}_{\mathbf{e}_q})$ has matrix Laplace transform given by  
\begin{align}\label{resolvent}
\begin{split}
&\quad{{E}^{0,i}(e^{- \alpha\overline{m}_{\mathbf{e}_q}- \beta \overline\xi_{\mathbf{e}_q}}, J_{\overline{m}_{\mathbf{e}_q}} = j) }\\
&={E}^{0,i}\Big[e^{-\alpha \mathcal{L}^{-1}_{\theta_q}-\beta\mathcal{H}_{\theta_q}}\mathbf{1}_{(J^+_{\theta_q} =j)}\Big]\\
&={E}^{0,i}\Big[\int_0^\infty du\,\mathbf{1}_{(J^+_u= j) } \Phi^+_{J^+_u}(q,0)e^{-\int_0^u \Phi^+_{J^+_s}(q,0)ds} e^{-\alpha \mathcal{L}^{-1}_u-\beta\mathcal{H}_u}\Big]\\
&= \int_0^\infty du\,  \Phi^+_{j}(q,0){E}^{0,i}\Big[e^{-\int_0^u \Phi^+_{J^+_s}(q,0)ds} e^{-\alpha \mathcal{L}^{-1}_{u}-\beta\mathcal{H}_{u}}\mathbf{1}_{(J^+_u= j) }\Big],
\end{split}
\end{align}
for $\alpha, \beta\geq 0$.
Note that  the final expectation above can be written in terms of the matrix Laplace exponent of $(\mathcal{L}^{-1}, \mathcal{H}, J^+)$ with a potential corresponding to $\text{diag}(\Phi^+_1(q,0), \cdots,\Phi^+_{|E|}(q,0))$, i.e.
\[
\kappa^+(q+\alpha, \beta) =\text{diag}(\Phi^+_1(q+\alpha,\beta) , \cdots,  \Phi^+_{|E|}(q+\alpha,\beta) )- Q^+\circ G^+(\beta), \qquad \alpha,\beta\geq 0.
\]
Indeed, one has,
\[
{E}^{0,i}\left[e^{-\int_0^u \Phi^+_{J^+_s}(q,0)ds} e^{-\alpha \mathcal{L}^{-1}_u-\beta\mathcal{H}_u}\mathbf{1}_{(J^+_u= j) }\right] = [e^{-\kappa^+(q+\alpha, \beta)}]_{i,j}.
\]
Continuing the computation in (\ref{resolvent}), we now have the following result. 
\begin{theorem}\label{preWHF}For $i,j\in E$, $\alpha, \beta\geq 0$ and $q>0$,
\begin{align}
{E}^{0,i}\big[e^{- \alpha\overline{m}_{\mathbf{e}_q}- \beta \overline\xi_{\mathbf{e}_q}}, J_{\overline{m}_{\mathbf{e}_q}} = j\big]
= \Phi^+_j(q,0)[\kappa^+(q+\alpha,\beta)^{-1}]_{i,j}.
\label{maxexp}
\end{align}
\end{theorem}

We can go a little further in our analysis of the previous section and note that, on the event $\{J^+_{\theta_q} =j\}$, the excursion $\epsilon_{J^+_{\theta_q}}$ is independent of $\{(\bar{L}^{-1}_t, H^+_t, J^+_t): t< \theta_q\}$. In particular, on $\{J^+_{\theta_q} =j\}$, we have that $(\overline{\xi}_{\mathbf{e}}, \overline{m}_{\mathbf{e}_q})$ is independent of $(\xi_{\mathbf{e}_q} - \overline{\xi}_{\mathbf{e}_q}, \mathbf{e}_q - \overline{m}_{\mathbf{e}_q} )$.

Duality allows us to conclude that on the event $\{J^+_{\theta_q} = j, J_{\mathbf{e}_q }= k\} =\{J_{\overline{m}_{\mathbf{e}_q}} = j, J_{\mathbf{e}_q} = k\}$ the pair $( \overline{\xi}_{\mathbf{e}_q}-\xi_{\mathbf{e}_q} , \mathbf{e}_q - \overline{m}_{\mathbf{e}_q} )$ is equal in law to the pair
$( \overline{\hat{\xi}}_{\mathbf{e}_q},  {\overline{\hat{m}}}_{\mathbf{e}_q} )$ on $\{\hat{J}_0 = k, \hat{J}_{\overline{\hat{m}}_{\mathbf{e}_q}}  =j\}$, where $\{(\hat\xi_s, \hat{J}_s): s\leq t\}: = \{ (\xi_{(t-s)-} -\xi_t, J_{(t-s)-}): s\leq t\}$, $t\geq 0$, is equal in law to the dual of $\xi$, $\overline{\hat{\xi}}_t = \sup_{s\leq t}\hat{\xi}_s$ and $\overline{\hat{m}} = \sup\{s\leq t: \overline{\hat\xi}_s = \hat{\xi}_t\}$.


From the previous section, we may now deduce that, for $i,j, k\in E$ and $\alpha, \beta \geq 0$,
\begin{eqnarray}
{E}^{0,i}\big[e^{- \alpha(\mathbf{e}_q - {m}_{\mathbf{e}_q})- \beta(\overline\xi_{\mathbf{e}_q}-\xi_{\mathbf{e}_q}) }, J_{{m}_{\mathbf{e}_q}} = j, \, J_{\mathbf{e}_q }  =k\big] 
&=& {E}^{0,k}\big[e^{-\alpha\overline{\hat{m}}_{\mathbf{e}_q}  -\beta \overline{\hat\xi}_{\mathbf{e}_q}  }, \hat{J}_{\overline{\hat{m}}_{\mathbf{e}_q}} = j\big] \notag\\
&=& \hat{E}^{0,k}\big[e^{-\alpha\overline{{m}}_{\mathbf{e}_q}  -\beta \overline{\xi}_{\mathbf{e}_q}  }, {J}_{\overline{{m}}_{\mathbf{e}_q}} = j\big] .
\label{dualequation}
 \end{eqnarray}

 We can also use the ideas above to prove the following technical lemma which will be of use later on.

 \begin{lemma} For all $j\in E$,
 \[
c : =\sum_{j\in E} \lim_{q\downarrow0}\frac{\Phi^+_j(q,0)\hat\Phi^+_j(q,0)}{q}
 \]
 exists in $(0,\infty)$ and, for each $j\in E$,
 \begin{equation}
 c_j := \lim_{q\downarrow0}\frac{\Phi^+_j(q,0)\hat\Phi^+_j(q,0)}{q}
 \label{timewh}
 \end{equation}
 exists in $[0,\infty)$.
 \end{lemma}
 \begin{proof}
 Write $\hat\kappa^+(\alpha,\beta)$ for the dual matrix exponent, that is, to $\hat{F}(z)$ what $\kappa^+(\alpha, \beta)$ is to $F(z)$. On the one hand, for all $i,k\in E$ and $\alpha> 0$,
  \begin{align*}
{E}^{0,i}\big[e^{-\alpha\mathbf{e}_q}, J_{\mathbf{e}_q}=k\big]
&= \Big[\int_0^\infty qe^{-(\alpha+q) t}e^{Qt}dt\Big]_{i,k}\\
&= q\big[\big((q+\alpha)I -Q\big)^{-1}\big]_{i,k}.
 \end{align*}
 On the other hand, from (\ref{dualequation}), for all $i,k\in E$ and $\alpha> 0$,
 \begin{align*}
{E}^{0,i}\big[e^{-\alpha\mathbf{e}_q}, J_{\mathbf{e}_q}=k\big] 
&=\sum_{j\in E}{E}^{0,i}\big[e^{-\alpha (\overline{m}_{\mathbf{e}_q} + \mathbf{e}_q -{\overline{m}}_{\mathbf{e}_q}    )}, J_{\overline{m}_{\mathbf{e}_q}}=j,\, J_{\mathbf{e}_q}=k\big]\\
&=\sum_{j\in E}\Phi^+_j(q,0)[\kappa^+(q+\alpha,0)^{-1}]_{i,j}\hat\Phi^+_j(q,0)[\hat\kappa^+(q+\alpha,0)^{-1}]_{k,j}
 \end{align*}
 Taking limits as $q\downarrow0$ it follows from continuity that 
 \[
 \big[(\alpha I - Q)^{-1}\big]_{i,k} = \sum_{j\in E}\lim_{q\downarrow0}\frac{\Phi^+_j(q,0)\hat\Phi^+_j(q,0)}{q}[\kappa^+(\alpha,0)^{-1}]_{i,j}[\hat\kappa^+(\alpha,0)^{-1}]_{k,j},
 \]
 where the limit on the right-hand side exists because the limit exits on the lefthand side. The statement of the theorem now follows. 
\end{proof}
 
 The next theorem below gives the Wiener--Hopf factorisation for MAPs. It is a natural consequence of Theorem \ref{preWHF} and a well-established method of splitting stochastic processes at their maximum. Some results already  exist in the literature in this direction, see for example Chapter XI of \cite{Asmussen} and \cite{Kaspi}, however, none of them are in an appropriate form for our purposes. 
 
 \begin{remark}\rm\label{c=1} As a consequence of the Wiener--Hopf factorisation,  it will turn out that the constants $c_j$, $j\in E$, are all strictly positive and may be taken to be equal to unity without loss of generality. 
 \end{remark}


\begin{theorem}\label{WHF}
For $z\in \mathbb{R}\backslash\{0\}$ and $\alpha\geq 0$, 
\[
\alpha I- F({\rm i}z) = \Delta_\pi^{-1}[\hat\kappa^+(\alpha, {\rm i}z)^{\rm T}]\Delta_{\pi}\kappa^+(\alpha, -{\rm i}z) .
\]
\end{theorem}

\begin{proof} We start by sampling $\xi$ over an independent and exponentially distributed time horizon denoted, as usual, by $\mathbf{e}_q$. By splitting at the maximum, applying duality and appealing to the identity (\ref{maxexp}), we have  for $\alpha\geq 0$
\begin{align*}
{E}^{0,i}\big[e^{ -\alpha\mathbf{e}_q +{\rm i}z\xi_{\mathbf{e}_q}} , J_{\mathbf{e}_q} = j \big]
&= \sum_{k\in E}{E}^{0,i}\big[e^{ -\alpha(\mathbf{e}_q - \overline{m}_{\mathbf{e}_q}  + \overline{m}_{\mathbf{e}_q}) +{\rm i}z\overline\xi_{\mathbf{e}_q}
} e^{{\rm i}z(\xi_{\mathbf{e}_q} - \overline{\xi}_{\mathbf{e}_q})}, J_{\overline{m}_{\mathbf{e}_q} } = k, \, J_{\mathbf{e}_q} = j\big]\\
&=\sum_{k\in E}{E}^{0,i}\big[e^{-\alpha \overline{m}_{\mathbf{e}_q} + {\rm i}z\overline\xi_{\mathbf{e}_q}
} , J_{\overline{m}_{\mathbf{e}_q} } = k\big]\frac{\pi_j}{\pi_k}\hat{{E}}^{0,j}\big[e^{-\alpha \overline{m}_{\mathbf{e}_q} -{\rm i}z\overline{\xi}_{\mathbf{e}_q}
} , J_{\overline{ m}_{\mathbf{e}_q} } = k\big]\\
&= \sum_{k\in E}
\Phi^+_k(q,0)[\kappa^+(q+\alpha,-{\rm i}z)^{-1}]_{i,k}\frac{\pi_j}{\pi_k}\hat\Phi^+_k(q,0)[\hat\kappa^+(q+\alpha,{\rm i}z)^{-1}]_{j,k}
\end{align*}
Noting that we can write the lefthand side above as 
$q[((q+\alpha)I - F({\rm i}z))^{-1}]_{i,j}$, we can divide by $q$ and take limits as $q\downarrow 0$ to find that
\[
[(\alpha I-F({\rm i}z))^{-1}]_{i,j} = \sum_{k\in E}c_k[\kappa^+(\alpha,-{\rm i}z)^{-1}]_{i,k}\frac{\pi_j}{\pi_k}[[\hat\kappa^+(\alpha,{\rm i}z)^{\rm T}]^{-1}]_{k,j},
\]
where we recall that the constants $c_k$, $k\in E$ were introduced in (\ref{timewh}). In matrix form, the above equality can be rewritten as 
\begin{equation}\label{invertthis}
(\alpha I-F({\rm i}z))^{-1} = \kappa^+(\alpha,-{\rm i}z)^{-1} \Delta_{c/\pi} [\hat\kappa^+(\alpha, {\rm i}z)^{\rm T}]^{-1}\Delta_\pi,
\end{equation}
where $\Delta_{c/\pi} = \text{diag}(c_1/\pi_1, \cdots c_{|E|}/\pi_{|E|} )$. Since all matrices are invertible except possibly $\Delta_{c/\pi}$ (on account of the fact that some of the constants $c_k$ may be zero), it follows that necessarily $c_k>0$ for all $k\in E$ and hence the matrix $\Delta_{c/\pi}$ is indeed invertible and is its inverse equal to $\Delta_{\pi/c}^{-1}$ (using obvious notation).
The proof is now completed by inverting the matrices on both left- and right-hand sides of (\ref{invertthis}) and noting that, without loss of generality, the constants $c_k$ may be taken as unity by choosing an appropriate normalisation of local time (which in turn means that the equality in  (\ref{timewh}) can be determined up to a multiplicative constant).
\end{proof}

A.5. {\bf Occupation formula.} The objective in this section is to use the preceding constructions to establish a key identity which is central to the analysis of real self-similar Markov processes in the main body of the text.  In order to state the main result, some more notation is needed. 

For $i,j\in E$ the potential measure $U^+_{i,j}$ on $[0,\infty)$ is defined by 
\begin{equation}
U^+_{i,j}(dx) ={E}^{0,i}\Big[\int_0^\infty \mathbf{1}_{(H^+_t \in dx, \, J^+_t= j)}\,dt\Big], \qquad x\geq 0.
\label{occupancyformula}
\end{equation}
Note that, for $\lambda> 0$,
\begin{equation}
\int_0^\infty e^{-\lambda x}U^+_{i,j}(dx)  
= \int_{0}^\infty{E}^{0,i} \big[e^{-\lambda H^+_t}, J^+_t = j\big]\,dt 
= [\kappa^+(0,\lambda)^{-1}]_{i,j}.
\label{LT}
\end{equation}
Moreover, it should also be noted that the non-lattice assumption on the process $\xi$ ensures that the measure $U^+_{i,j}$ is diffuse on $(0,\infty)$; see the discussion at the end of A.3 as well as the proof of Theorem 5.4  in \cite{kyprianou} in the L\'evy case for guidance.
We can define by analogy the measures $\hat{U}^+_{i,j}$, $i,j\in E$, for to the dual process $\hat\xi$. The reader might also want to recall the definitions of $\tau^-_0$ and $\tau^+_0$ from \eqref{lala}.

\begin{theorem}\label{bertoin-spitzer}
There exist non-negative constants $c_j$, $j\in E$, satisfying $\sum_{j\in E}c_j>0$ such that for all bounded measurable $f: \mathbb{R}\to [0,\infty)$ and $x>0$,
\[
{E}^{x,i}\Big[\int_0^{ {\tau}^-_0}f(\xi_t)\mathbf{1}_{(J_t =k)}\,dt\Big] = \sum_{j\in E}c_j\int_{y\in[0,\infty)}\int_{z\in[0,x]}U^+_{i,j}(dy)\hat{U}^+_{k,j}(dz) f(x+y -z).
\]
\end{theorem}
\begin{proof}
Start by noting that 
\begin{align}
&\quad\lefteqn{
 E^{x,i}\Big[\int_0^{ {\tau}^-_0}e^{-qt}f(\xi_t)\mathbf{1}_{(J_t =k)}dt\Big]}&\notag\\
 &=
\frac{1}{q}E^{x,i}\big[f(\xi_{\mathbf{e}_q})\mathbf{1}_{(J_{\mathbf{e}_q} =k)}, \mathbf{e}_q < {\tau}^-_0\big]
\notag\\
&=\frac{1}{q}\sum_{j\in E}{E}^{x,i}\big[f( \overline{\xi}_{\mathbf{e}_q}-(\overline{\xi}_{\mathbf{e}_q} -\xi_{\mathbf{e}_q} ))\mathbf{1}_{(J_{\overline{m}_{\mathbf{e}_q}} =j)}\mathbf{1}_{(J_{\mathbf{e}_q} =k)}, \mathbf{e}_q < {\tau}^-_0\big]
\notag\\
&=\int_{y\in[0,\infty)}\int_{z\in[0,x]} f(x+y -z) \sum_{j\in E}\frac{1}{q}{P}^{0,i}\big(\overline{\xi}_{\mathbf{e}_q} \in dy, J_{\overline{m}_{\mathbf{e}_q}}=j\big){P}^{0,i}\big(\overline{\xi}_{\mathbf{e}_q} - \xi_{\mathbf{e}_q} \in dz, J_{\overline{m}_{\mathbf{e}_q}} = j, \, J_{\mathbf{e}_q} =k\big)\notag\\
&=\int_{y\in[0,\infty)}\int_{z\in[0,x]} f(x+y -z) \sum_{j\in E}\frac{1}{q}{P}^{0,i}\big(\overline{\xi}_{\mathbf{e}_q} \in dy, J_{\overline{m}_{\mathbf{e}_q}}=j\big){P}^{0,k}\big(\overline{\hat\xi}_{\mathbf{e}_q} \in dz,  J_{\overline{\hat{m}}_{\mathbf{e}_q}} =j\big).
\label{product}
\end{align}
Next, with the help of (\ref{maxexp}),
\begin{align*}
&\quad\lefteqn{
\int_{[0,\infty)}e^{-\lambda y - \mu z}\sum_{j\in E}\frac{1}{q}{P}^{0,i}\big(\overline{\xi}_{\mathbf{e}_q} \in dy, J_{\overline{m}_{\mathbf{e}_q}}=j\big){P}^{0,k}\big(\overline{\hat\xi}_{\mathbf{e}_q} \in dz,  J_{\overline{\hat{m}}_{\mathbf{e}_q}} =j\big)}&&\\
&= \sum_{j\in E}\frac{\Phi^+_j(q,0)\hat\Phi^+_j(q,0)}{q}[\kappa^+(q,\lambda)^{-1}]_{i,j}[\hat\kappa^+(q,\mu)^{-1}]_{k,j},
\end{align*}
for $\lambda, \mu>0$.
Taking account of (\ref{LT}), it follows with the help of Lebesgue's Continuity Theorem for Laplace transforms that, in the vague sense, the product measure on the right-hand side of (\ref{product}) satisfies
\begin{align*}
\lim_{q\downarrow0} \sum_{j\in E}\frac{1}{q}{P}^{0,i}\big(\overline{\xi}_{\mathbf{e}_q} \in dy, J_{\overline{m}_{\mathbf{e}_q}}=j\big){P}^{0,k}\big(\overline{\hat\xi}_{\mathbf{e}_q} \in dz,  J_{\overline{\hat{m}}_{\mathbf{e}_q}} =j\big) = \sum_{j\in E}c_j U^+_{i,j}(dy)\hat{U}^+_{k,j}(dz).
\end{align*}
The result now follows for non-negative compactly supported, bounded measurable $f\geq 0$ and hence, appealing to standard monotonicity arguments, one can upgrade the result to deal with  bounded measurable $f\geq 0$.
\end{proof}

\bigskip

A.6. {\bf Markov Renewal theory.} 
The measures $U^+_{i,j}$ play an analogous role to the potential measure $U$ of the ascending ladder process for a L\'evy process, which can also be seen as a renewal measure. For example, using an analogue of the compensation formula for Cox processes, it is straightforward to deduce that, for $a,x>0$,  
\begin{align}\label{cts}
\begin{split}
&\quad P^{0,i}(\xi_{\tau^+_a}-a > x, J^+_{{\tau}^+_a} = j)\\
&=\int_{[0,a)} U^+_{i,j}(dy)n_j(\epsilon_\zeta >a-y+x, J_\zeta =j)
+ \sum_{k\neq j}\int_{[0,a)}q^+_{k,j}U^+_{i,k}(dy)(1-F^+_{k,j}(a-y +x)).
\end{split}
\end{align}
It is worth noting here that the fact that $U^+_{i,j}$ is diffuse on $(0,\infty)$ ensures that the right-hand side above is continuous in $x$.

There is a relatively wide body of literature concerning Markov additive renewal theory; see for example  \cite{lalley}, \cite{kesten} and \cite{Alsmeyer}. Although mostly dealt with for the case of discrete-time, we can nonetheless identify the following renewal-type theorem for the non-lattice measures $U^+_{i,j}$.
\begin{theorem}\label{MRT}
The family $\{\xi_{\tau^+_a}-a :a>0\}$ of overshoots converges in distribution under $P^{0,i}$ for every $i\in E$ if and only if $$E^{0,\pi}[H^+_1] : = \sum_{i\in E}\pi_i E^{0,i}[H^+_1]<\infty$$ and in that case the following hold:
\begin{itemize}
\item[(i)] For all $i,j\in E$, 
\[
\lim_{x\to\infty}\frac{U^+_{i,j}(x)}{x} = \frac{\pi_j}{E^{0,\pi}[H^+_1]}.
\]
\item[(ii)] In the spirit of the Key Renewal Theorem, for $\alpha>0$ and $i,j\in E$,
\[
\lim_{y\to\infty}\int_{[0,y]} e^{-\alpha(y -z)}{U}^+_{i,j}(dz) 
 = \frac{\pi_j}{\alpha E^{0,\pi}[H^+_1]}.
\] 
\item[(iii)]For $x>0$ and $i,j \in E$, 
\begin{align*}
\nu(dx,\{j\})&:=\wlim_{a\to\infty} P^{0,i}(\xi_{\tau^+_a}-a \in dx, J^+_{{\tau}^+_a} = j)\\
&= \frac{1}{E^{0,\pi}[H^+_1]}\Big[ \pi_jn_j(\epsilon_{\zeta} > x, J_\zeta =j) + \sum_{k\neq j}\pi_k q^+_{k,j} (1-F^+_{k,j}(x))\Big]dx,
\end{align*}
where   $F^+_{k,j}$ is the distribution whose Laplace transform is $G^+_{k,j}$.
\end{itemize}
For all limits above, we interpret the right-hand side as zero when $E^{0,\pi}[H^+_1]  = \infty$. In particular, this means that the overshoot distributions diverge to an atom at $+\infty$ and are not tight.
\end{theorem}
Parts (i) and (iii) are the continuous-time analogue of the  Markov Additive Renewal Theorem in \cite{lalley}, whereas part (ii) is the continuous time analogue of the version of the Markov Additive Renewal Theorem in \cite{kesten}.
\bigskip

A.7. {\bf Harmonic functions} 
The main objective of this section is to prove a result which identifies a harmonic function for the process $(\xi, J)$ when killed on entering $(-\infty,0)\times E$. In the forthcoming analysis we use $\underline{L}_t$ to denote $\sum_{k\in E}\underline{L}^{(k)}_t$, the sum of local times of $\underline{\xi} - \xi$ at $(0,k)$, $k\in E$, where $\underline\xi_t = \inf_{s\leq t}\xi_s$, $t\geq 0$. Moreover,  similarly to previous sections in this Appendix, we work with $H^-_t : = \xi_{\underline{L}^{-1}_t}$ and  $J^-_t = J_{\underline{L}^{-1}_t}$, for all $t$ such that $\underline{L}^{-1}_t<\infty$ and otherwise the pair $H^-_t$ and $J^-_t$ are both assigned the value $\infty$. Furthermore, define 
\begin{align*}
U^-_i(x)= {E}^{0,i}\Big[\int_0^\infty\mathbf{1}_{(H^-_t \leq x)}dt\Big], \qquad x\geq 0,
\end{align*}
then the following theorem holds:

\begin{theorem}\label{hfunction}
For all $i\in E$ and $x> 0$, 
\[
U^-_{J_t}( \xi_t)\mathbf{1}_{(t<\tau^-_0)}, \qquad t\geq 0,
\]
is a ${P}^{x,i}$-martingale if and only if $\overline{\xi}-\xi$ is recurrent at zero; that is to say the Markov process $(\overline{\xi} - \xi , J)$ is recurrent at $(0,k)$ for some (and hence all) $k\in E$.
\end{theorem}

We start by proving a preliminary lemma giving us an important fluctuation identity.
To this end, define for $q>0$ the measure,
\begin{align*}
{ }^{q}U^-_{i,j}(dx)  = {E}^{0,i}\Big[\int_0^\infty e^{-q\underline{L}^{-1}_t}\mathbf{1}_{(H^-_t \in dx, \, J^-_t= j)}\,dt\Big],\qquad x\geq 0,
\end{align*}
and set
\begin{align*}
 {}^{q}U^-_{i}(x) =\sum_{j\in E}    {}^{q}U^-_{i,j}(x), \qquad x\geq 0.
\end{align*}
Recall that $\mathbf{e}_q$ denotes an independent exponentially distributed random variable with rate $q>0$ and $\tau^-_0 := \inf\{t> 0 : \xi_t<0\}.$
Let $\underline{n}_i$ be the excursion measures of $\underline{\xi}-\xi$ from the point $(0,i)$, $i\in E$. For convenience, let us assume that each of the subordinators $[\underline{L}^{(k)}]^{-1}_t$, $k\in E$ have no drift component. The corresponding forthcoming computation when this is not the case is a straightforward modification, e.g. in the spirit of, for example, the proof of Lemma VI.8 of \cite{Bertoin}. 

If we  mark the excursion  from the minimum indexed by local time $t>0$ with an independent exponentially distributed random variables, say $\mathbf{e}_q^{(t)}$, then using the compensation formula for the Cox process of excursions of $\underline{\xi} -\xi$ from 0, we have 
\begin{align*}
&\quad{
{P}^{x,i}\big( {\tau}^-_0>\mathbf{e}_q, J_{\underline{m}_{\mathbf{e}_q}} =j\big)
}\\
&={E}^{0,i}\Bigg[{E}^{0,i}\Big[\sum_{t\geq 0} \mathbf{1}_{(H^-_{t-}\leq x, \, \Delta \underline{L}^{-1}_s <\mathbf{e}^{(s)} \, \forall s<t)} \mathbf{1}_{(\Delta \underline{L}^{-1}_t >\mathbf{e}^{(t)} , \,J^-_t = j)}\,\Big|\, \sigma(J_u: u\geq 0) \Big]\Bigg]\\
&={E}^{i,0} \Big[ \int_0^\infty dt\cdot  
 \mathbf{1}_{(H^-_{t-}\leq x, \, \Delta \underline{L}^{-1}_s <\mathbf{e}^{(s)} \, \forall s<t)}
\Big]\underline{n}_j(\zeta >\mathbf{e}_q)\\
&={E}^{0,i} \Big[ \int_0^\infty dt\cdot  
 e^{-q\underline{L}^{-1}_t}\mathbf{1}_{(H^-_{t-}\leq x)}
\Big]\underline{n}_{j}(1-e^{-q\zeta})\\
&={}^{q}U^-_{i} (x)\Phi_{j}^-(q,0),
\end{align*}
where $\Phi_{j}^-(q,0): = \underline{n}_{j}(1-e^{-q\zeta})$ is a notational choice that, by analogy, respects the definition of $\Phi^+_j(\alpha, \beta)$ given in Section A.2. 
In the second and third equalities above, the letter $\zeta$ denotes the canonical excursion length. In conclusion, we have established the following lemma.

\begin{lemma} For all $i,j\in E$ and $x>0$,
\begin{equation}
{P}^{x,i}\big({\tau}_0^->\mathbf{e}_q, J_{\underline{m}_{\mathbf{e}_q}} =j\big)= {}^{q}U^-_{i}(x)\Phi_j^-(q, 0).
\label{q-harmonic}
\end{equation}
\end{lemma}


We now return to the proof of Theorem \ref{hfunction}.

\begin{proof}[Proof of Theorem \ref{hfunction}]
Thanks to the Markov property, it suffices to prove that, for all $i\in E$ and $x>0$,
\[
{E}^{x,i}[U^-_{J_t}( \xi_t)\mathbf{1}_{(t<{\tau}^-_0)}] = U^-_i(x).
\]
To proceed, we use ideas from \cite{CD} and Chapter 13 of \cite{kyprianou}. With the help of monotone convergence, we have that 
\begin{align*}
&\quad{E}^{x,i}[U^-_{J_t}(\xi_t),  \, t< {\tau}^-_0]\\ 
&= \lim_{q\downarrow 0}  {E}^{x,i}\Bigg[ \mathbf{1}_{(t< {\tau}^-_0)}\frac{{P}^{\xi_t,J_t}( {\tau}^-_0>\mathbf{e}_q)}{\sum_{j\in E}\Phi_j^-(q, 0) }\Bigg]\\
&= \lim_{q\downarrow 0} \frac{1}{\sum_{j\in E}\Phi_j^-(q, 0) } {P}^{x,i}\left[\left. {\tau}^-_0>\mathbf{e}_q\right|\mathbf{e}_q >t\right]\\
&=\lim_{q\downarrow 0} \left[e^{qt}\frac{ {P}^{x,i}( {\tau}_0^->\mathbf{e}_q)}{\sum_{j\in E}\Phi_j^-(q, 0) }
-e^{qt}\int_0^t q e^{-qs} \frac{ {P}^{x,i}( {\tau}^-_0>s)}{\sum_{j\in E}\Phi_j^-(q, 0) }
ds\right]\\
&=U^-_i(x) - \lim_{q\downarrow 0} \frac{q}{\sum_{j\in E}\Phi_j^-(q, 0) }\int_0^t    {P}^{x,i}( {\tau}_0^->s)ds.
\end{align*}
The proof is complete as soon as we can show that the limit preceding the integral term is equal to  zero. To this end, note that for each $j\in E$,
\[
\lim_{q\downarrow 0}\frac{1}{q} \Phi^-_j(q, 0) = \Phi^{-\prime}_j(0,0) = {E}^{0,k}\big[[\underline{L}^{(k)}]^{-1}_1\big]\in(0,\infty].
\]
We want to show that the expectation on the right-hand side above to be $+\infty$ as a consequence of the fact that $0$ is recurrent for $\overline\xi-\xi$. 
Appealing  to (\ref{maxexp}), we have 
\[
{E}^{0,i}\big[e^{- \alpha\underline{m}_{\mathbf{e}_q} },  J_{{m}_{\mathbf{e}_q}} = k\big] =\Phi^-_k(q,0)[ \Phi^-(q+\alpha,0)^{-1}]_{i,k}.
\]
Duality dictates that 
\[
{E}^{0,i}\big[e^{- \alpha\underline{m}_{\mathbf{e}_q} },  J_{{m}_{\mathbf{e}_q}} = k\big]  =\frac{\pi_k}{\pi_i}\hat{{E}}^{0,k}\big[e^{- \alpha\overline{m}_{\mathbf{e}_q} },  J_{\overline{m}_{\mathbf{e}_q}} = i\big], 
\]
which tells us that 
\[
\frac{1}{q}\Phi^-_k(q,0)[ \kappa^-(q+\alpha,0)^{-1}]_{i,k} = \frac{\pi_k}{\pi_i}\frac{1}{q}\hat\Phi^+_j(q,0)[ \hat\kappa^+(q+\alpha,0)^{-1}]_{i,k}.
\]
In turn, this  means that $\lim_{q\downarrow 0}\Phi^-_k(q, 0)/q$ and $\lim_{q\downarrow 0}\hat\Phi^+_k(q, 0)/q$ are simultaneously (in)finite. Note that both have limits because they are Bernstein functions.

Now recall from (\ref{timewh}) that, since, 
\begin{equation}
c_k:= \lim_{q\downarrow0}\frac{\Phi^+_k(q,0)\hat\Phi^+_k(q,0)}{q},
\label{choose=1}
\end{equation}
 it follows that $\lim_{q\downarrow 0}\Phi^-_k(q, 0) /q=\infty$ if  $\Phi^+_k(0,0) = 0$. However, the assumption that $\overline\xi-\xi$ is recurrent at $0$ ensures that $\Phi^+_k(0,0) = 0$ for all $k\in E$.

In conclusion, we have that, under the assumption that $\overline\xi-\xi$ is recurrent at 0,  the term 
\[
\lim_{q\downarrow 0} \frac{q}{\sum_{j\in E}\Phi_j^-(q, 0) } = 0,
\]
and subsequently the claim of the theorem is proved.
\end{proof}


\bigskip 

A.8. {\bf Conditioning to stay positive.} It turns out that the harmonic function $U^-_j(x)$, $j\in E$, $x>0$ corresponds to the $h$-function that appears in the Doob $h$-transform corresponding to the process $\xi$ conditioned to stay positive. 

Let  $A\in\mathcal{F}_t : = \sigma((\xi_s, J_s): s\leq t)$ and assume that 0 is recurrent for $\overline\xi -\xi$. 
 Appealing to the Markov and lack of memory properties, we have 
\[
\lim_{q\downarrow 0}  {P}^{x,i} (A, t<{\mathbf{e}_q } \,|\, {\tau}^-_0>{\mathbf{e}_q })
=
\lim_{q\downarrow 0} {E}^{x,i} \left[\mathbf{1}_{(A, \, t< {\tau}^-_0 < \mathbf{e}_q)}\frac{{P}^{\xi_t, J_t}  ({\tau}^-_0>{\mathbf{e}_q }) }{ {P}^{x,i} ({\tau}^-_0>{\mathbf{e}_q })} \right].
\]
Next  note that, for all $q<q_0$,
\[
\frac{{P}^{\xi_t, J_t}  ({\tau}^-_0>{\mathbf{e}_q }) }{ {P}^{x,i } ({\tau}^-_0>{\mathbf{e}_q })}  =\frac{{}^{q}U^-_{J_t}(\xi_t)}{{}^{q}U^-_{i}(x)}\leq \frac{{}^{q_0}U^-_{J_t}(\xi_t)}{U^-_{i}(x)}.
\]
Hence, by dominated convergence, we have that 
\begin{align*}
\lim_{q\downarrow 0}  {P}^{x,i} \big(A, t<{\mathbf{e}_q } \,\big|\, {\tau}^-_0>{\mathbf{e}_q }\big)
= {E}^{x,i} \Big[\mathbf{1}_{(A, \, t< {\tau}^-_0 )}\frac{U^-_{J_t}(\xi_t) }{ U^-_i(x)}\Big].
\end{align*}
In conclusion, we have the following theorem which confirms the existence of the law of $(\xi, J)$ with $\xi$ conditioned to stay positive.
\begin{theorem}
Suppose that $0$ is recurrent for $\overline\xi-\xi$. Then there exists a family of probability measures on the Skorokhod space, say $\mathbb{P}^\uparrow_{x,i}$, defined via the Doob $h$-transform
\[
\left.\frac{d{P}^{x,i,\uparrow}}{d {P}^{x,i}}\right|_{\mathcal{F}_t} =
\frac{U^-_{J_t}(\xi_t)}{U^-_i(x)}\mathbf{1}_{(t< {\tau}^-_0)}, \qquad t\geq 0, i\in E, x>0,
\]
such that, for all $A$ in $\mathcal{F}_t$, 
\[
{P}^{x,i,\uparrow}(A) = \lim_{q\downarrow 0}  {P}^{x,i} (A, t<{\mathbf{e}_q } \,|\, {\tau}^-_0>{\mathbf{e}_q }).
\]
\end{theorem}
\begin{remark}\label{dualcond}
	Setting 
	\begin{align}\label{db}
	U^+_i(x)= {E}^{0,i}\Big[\int_0^\infty\mathbf{1}_{(H^+_t \leq -x)}dt\Big],\quad x<0,
	\end{align} 
	the above discussion applied to the MAP $(-\xi,J)$ implies that $U^+_{J_t}( \xi_t)\mathbf{1}_{(t<\tau^+_0)}$ is a martingale and the $h$-transformed law 
	\begin{align*}
\left.\frac{d{P}^{x,i,\downarrow}}{d {P}^{x,i}}\right|_{\mathcal{F}_t} =
\frac{U^+_{J_t}(\xi_t)}{U^+_i(x)}\mathbf{1}_{(t<\tau^+_0)}, \qquad t\geq 0, i\in E, x<0,
\end{align*}

	  is the MAP conditioned to be negative.

\end{remark}

For the proof of Lemma \ref{lll} we shall need that conditioned MAPs tend to infinity. In the context of L\'evy processes many proofs exist for analogue of the next lemma. Those proofs are consequences of complicated pathwise constructions for the conditioned processes that we do not want to repeat for the setting of MAPs. Instead we give a simple argument based on potential calculations only. The argument is inspired by more explicit calculations for spectrally negative L\'evy processes in Lemma VII.12 of \cite{Bertoin}.

\begin{proposition}\label{cond}  For each $x<0$ and $i\in E$, we have that $P^{x,i, \downarrow}(\lim_{t\to-\infty}\xi_t =- \infty) =1$.
\end{proposition}

\begin{proof}
First note that, for all $z<x<0$ and $i\in E$, with the help of \eqref{ddd},
\begin{align}
	E^{x,i, \downarrow}\left[\int_0^\infty \mathbf{1}_{(\xi_t \geq z)}dt\right]
	=U^\downarrow((x,i),([z,0],E))
	=\sum_{j\in E}
	\int_{[z,0]} \frac{{U}^+_j(y)}{{U}^+_i(x)}U^\dagger((x,i), (dy,\{j\})),
	\label{leif}
\end{align}
where $U^\dagger((x,i), (dy,\{j\}))$ is the potential measure of the process $(\xi, J)$ killed when $\xi$ first enters  $(0,\infty)$. Since $U^+$ is locally bounded the righthand side can be estimated from above by $\frac{C}{U^+(x)} U^\dag((x,i),([0,z],E))$ which is finite by Theorem \ref{bertoin-spitzer} applied with $f\equiv 1$ and using the local boundedness of the appearing potential measures of the ladder processes. This implies that 
\begin{align}\label{ca}
P^{x,i,\downarrow}(\tau^-_z<\infty)=1,\quad \text{for all }z<x<0, i\in E.
\end{align}
 Otherwise, the trajectory of $\xi$ is bounded from below by $z$ with positive probability under $P^{x,i, \downarrow}$ and, hence, $\int_0^\infty\mathbf{1}_{(\xi_t \geq z)}dt=\infty$ with positive probability. But then the left-hand side of \eqref{leif} would be infinite, giving a contradiction.

\medskip

Next, we show that
\begin{align}\label{cb}
\lim_{z\to-\infty}P^{z,i,\downarrow}(\xi_t < a \text{ for all }t\geq 0) =1,\quad \text{for all } a<0, i\in E.
\end{align}
To see this, define $\tau_{[a,0]} = \inf\{t>0: \xi_t \in[a,0]\}$. Use the change of measure in Remark \ref{dualcond} to note that, for $z<a$,
\begin{align*}
P^{z,i,\downarrow}(\text{there is }t\geq 0\text{ such that }\xi_t \geq a)
&=E^{z,i,\downarrow}[\mathbf 1_{(\tau_{[a,0]}<\infty)}]\\
&=E^{z,i,\dag}\left[\frac{U^+_{J_{\tau_{[a,0]}}} (\xi_{\tau_{[a,0]}})}{U^+_i(z)} \mathbf{1}_{(\tau_{[0,a]} < \tau^+_0)}\mathbf 1_{(\tau_{[0,a]}<\infty)}\right].
\end{align*}
Using the monotonicity of $z\mapsto U^+_i(z)$ from the definition \eqref{db} the right-hand side can be bounded from above by
\begin{align*}
	\frac{\max_{j\in E}U^+_j(a)}{U^+_i(z)}P^{z,i,\dag}(\tau_{[0,a]} < \tau^+_0).
\end{align*}
Finally, since $\lim_{z\to -\infty} U^+_i(z)=+\infty$, \eqref{cb} is proved.
\smallskip

The claim of the proposition now follows from the strong Markov property applied to $\tau^-_z$, which is finite by \eqref{ca}, and \eqref{cb}.
\end{proof}

A.9. {\bf Laws of large numbers.} Similarly to the case of L\'evy process, it is known that  a MAP $(\xi,J)$ grows linearly, meaning that 
\begin{align}\label{mmm}
\lim_{t\to\infty}\frac{\xi_t}{t} = E^{0,\pi}[\xi_1]
\end{align}
provided
\begin{align*}
E^{0,\pi}[\xi_1] = \sum_{i\in E}\pi_i E^{0,i}[\xi_1]
\end{align*}
is defined. Moreover, when $E^{0,\pi}[\xi_1]$ is defined there is a trichotomy which dictates whether $(\xi,J)$ drifts to $+\infty$, $-\infty$ or oscillates accordingly as $E^{0,\pi}[\xi_1]>0$, $<0$ or $=0$, respectively.
See for example Chapter XI of \cite{Asmussen}.\smallskip

We fix a state $k\in E$ and consider the MAP at the discrete set of return times of $J$ to $k$.  Let $\sigma_0=\inf\{t\geq 0\colon J_t=k\}$ and inductively define, for $n\in\N$,
\begin{align}\label{eq1311-1}
\sigma_{n+1}=\inf\{t> \sigma_{n} \colon J_t=k \text{ and }\exists s\in(\sigma_{n-1},t)\text{ with } J_s\not=k\}.
\end{align}
The skeleton $(\xi_{\sigma_n})_{n\in\N_0}$ is a Markov chain. The following theorem relates the law of large numbers to moments of the underlying L\'evy processes and transition jumps appearing in Proposition \ref{p1} and gives an identity that is crucial for the next section.

\begin{theorem}\label{theo1311-1}
The following statements are equivalent for a MAP $(\xi,J)$:
\begin{enumerate}
\item[(i)] $\xi_1$ has finite absolute mean for one (any) starting distribution with $\xi_0=0$.
\item[(ii)] $\xi_{\sigma_1}$ has finite absolute mean when started in $(0,k)$.
\item[(iii)] The L\'evy processes $\xi^i$ have finite absolute moment and any $\Delta_{i,j}$ with $q_{i,j}>0$ has finite absolute moment.
\item[(iv)] $\lim_{t\to\infty} \frac{\xi_t}{t}$ exists almost surely for one (any) starting distribution.
\end{enumerate}
Under (i) to (iv) we have
\begin{align}\label{nn}
	\lim_{t\to\infty} \frac{\xi_t}{t}= E^{0,\pi}[\xi_1]= \frac{E^{0,k}[\xi_{\sigma_1}]}{E^{0,k}[\sigma_1]},\quad k\in E.
\end{align}
\end{theorem}

\begin{proof}  Throughout the proof, we shall use the fact that for any L\'evy process $\{\eta_t :t\geq 0\}$
\begin{align*}
	E[|\eta_s|]<\infty\text{ for some }s>0\quad &\Longleftrightarrow \quad E[|\eta_t|]<\infty\text{ for all }t\geq 0\\
	&\Longleftrightarrow \quad E[\sup_{s\leq t}|\eta_s|]<\infty\text{ for all }t\geq 0.
\end{align*} See Theorem 25.18 of Sato \cite{Sato} for a proof.\\
\smallskip

\underline{(iii) $\Rightarrow$ (i):}
Note that for a fixed distribution $P^{0,\mu}$, by Proposition \ref{p1}, the distribution of $\xi_1$ is identical to the law of
\begin{align}\label{eq0311-1}
\sum_{i\in E} \xi^i_{t_i(1)} + \sum_{i\not= j} \sum_{\ell=1}^{n_{i,j} (1)} \Delta_{i,j}^{\ell},
\end{align}
where $t_i(1)$ denotes the time $J$ spends in state $i$, and $n_{i,j}(1)$  the number of jumps of $J$ from $i$ to $j$ over the time interval $[0,1]$ and, for each $i,j\in E$ such that $i\neq j$,  $\{\Delta_{i,j}^\ell : \ell\geq 1\}$ are iid copies of $\Delta_{i,j}$. Since the expected number of total jumps is finite,  the triangle inequality  shows that (iii) implies (i).\\ \smallskip

\underline{(i) $\Rightarrow$ (iii):} By considering the event that the first jump away from the initial state $i\in E$ occurs after time $1$, we have that $E^{0,i}[|\xi_1|]\geq E^{0,i}[|\xi^i_1|]\exp\{-|q_{i,i}|\}$, thereby showing that each of the pure-state L\'evy processes $\xi^i$, $i\in E$, have finite absolute moment. Now  consider the event that the first jump of the Markov chain $J$  occurs before time 1 and the second jump occurs after time $1$. In that case, we have 
\[
 \int_0^1|q_{i,i}|e^{-|q_{i,i}| t} \sum_{j\neq i}\frac{q_{i,j}}{|q_{i,i}|}e^{-|q_{j,j}| (1-t)} E^{0,i}[|\xi^i_t + \Delta_{i,j}+\xi^j_{1-t}| ]
dt<E^{0,i}[|\xi_1|]<\infty.
\]
This tells us that for each $j\in E$,  Lebesgue almost everywhere in $[0,1]$, 
\begin{equation}
E^{0,i}[|\xi^i_t + \Delta_{i,j}+\xi^j_{1-t}| ]
<\infty.
\label{(train)}
\end{equation} 
For a given $j\in E$ with $j\neq i$, fix such   a $t\in[0,1]$ and note that 
\begin{eqnarray*}
E^{0,i}[|\Delta_{i,j}| ] &=& E^{0,i}[|\xi^i_t + \Delta_{i,j}+\xi^j_{1-t} -\xi^i_t -  \xi^j_{1-t} | ]\\
&\leq& 
E^{0,i}[|\xi^i_t + \Delta_{i,j}+\xi^j_{1-t}| ]+ E^{0,i}[|\xi^i_t|] + E^{0,i}[|\xi^j_{1-t} |] \\
&<&\infty,
\end{eqnarray*}
where the final inequality follows by \eqref{(train)}, the previously established fact that $E^{0,i}[|\xi^i_{1} |]<\infty$ for $i\in E$ and the opening remark at the beginning of this proof.\\

\smallskip

\underline{(i) $\Rightarrow$ (ii):} We can identify the distribution of $\xi_{\sigma_1}$ with that of
\begin{eqnarray}
\sum_{i\in E} \xi^i_{t_i(\sigma_1)} + \sum_{i\not= j} \sum_{\ell=1}^{n_{i,j} (\sigma_1)} \Delta_{i,j}^{\ell},
\label{tired}
\end{eqnarray}
where $t_i(\sigma_1)$ denotes the time $J$ spends in state $i$, and $n_{i,j}(\sigma_1)$  the number of jumps of $J$ from $i$ to $j$ over the time interval $[0,\sigma_1]$ and, for each $i,j\in E$ such that $i\neq j$,  $\{\Delta_{i,j}^\ell : \ell\geq 1\}$ are iid copies of $\Delta_{i,j}$ (also independent of $n_{i,j}(\sigma_1)$, which depends only on the chain $J$). Note that $t_i(\sigma_1)$ is a random sum of an independent, geometrically distributed number of  independent exponential random variables that depend only on $J$, so that $E^{0,k}[|\xi^i_{t_i(\sigma_1)}|]<\infty$ whenever (iii) holds.
Having already shown the equivalence of (i) and (iii), it follows from the triangle inequality and the distributional equivalence in (\ref{tired}) that (i) implies (ii).\\ \smallskip
 
 \underline{(ii) $\Rightarrow$ (iii):}
  On the event that the sojourn of $J$ from $k$ consists of a first jump from $k$ to $j\neq k$, followed by a jump back to $k$, written $\{k\rightarrow j\rightarrow k\}$, we can write
\[
\xi_{\sigma_1}  = \xi^{k}_{\mathbf{e}_{|q_{k,k}|}} + \Delta_{k,j} + \xi^j_{\mathbf{e}_{|q_{j,j} |}} + \Delta_{j,k},
\]
where, for $i\in E$, $\mathbf{e}_{|q_{i,i}|}$ is an independent exponentially distributed random variable with rate $|q_{i,i}|$ and the sum on the right-hand side above consists of four independent random variables. This means that 
\begin{align}\label{ell}
\infty> E^{0,k}[|\xi_{\sigma_1}|] \geq E^{0,k}[|\xi_{\sigma_1}|\mathbf{1}_{\{k\rightarrow j\rightarrow k\}}]
=\E[|\xi^{k}_{\mathbf{e}_{|q_{k,k}|}} + \Delta_{k,j} + \xi^j_{\mathbf{e}_{|q_{j,j} |}} + \Delta_{j,k}  |]
\end{align}
if we denote by $\E$ the product space of the two L\'evy processes, two transition jumps and two exponential variables.\\
From the aforesaid independence we can deduce (iii). As a first step integrate out the final three summands on the righthand side of \eqref{ell}:
\begin{align*}
E^{0,k}[|\xi_{\sigma_1}|\mathbf{1}_{\{k\rightarrow j\rightarrow k\}}]
= \int_\R \int_0^\infty \int_0^\infty  \E[|\xi^k_{\mathbf{e}_{|q_{k,k}|}}+a+b+c|]|\P(\Delta_{k,j} \in da,\xi^j_{\mathbf{e}_{|q_{j,j} |}}\in db,\Delta_{j,k}\in dc).
\end{align*}
The left-hand side is finite and non-zero so there is some $x\in \R$ with $\E[|\xi^k_{\mathbf{e}_{|q_{k,k}|}}+x|]<\infty$. Integrating out the independent exponential time and using that $\xi^k$ is a L\'evy process implies that $\E[|\xi^k_1|]<\infty$ and $\E[|\xi^k_{\mathbf{e}_{|q_{k,k}|}}|]<\infty$ (compare the remark at the beginning of the proof and also note that $\E[|\xi^k_1|]<\infty$ if and only if $\E[|\xi^k_1+x|]<\infty$ for any $x\in\R$).

Similarly, we find that $\E[|\xi^j_1|]<\infty$ and
$\E[|\xi^j_{\mathbf{e}_{|q_{j,j}|}}|]<\infty$. Using the triangle inequality implies
\begin{align*}
\E[ |\Delta_{k,j} +  \Delta_{k,j}  |]
\leq \E[|(\Delta_{k,j} + \Delta_{j,k})+(\xi^{k}_{\mathbf{e}_{|q_{k,k}|}} + \xi^j_{\mathbf{e}_{|q_{j,j} |}} )  |]+ \E[|\xi^{k}_{\mathbf{e}_{|q_{k,k}|}}|]+\E[| \xi^j_{\mathbf{e}_{|q_{j,j} |}} |]
\end{align*}
and the right-hand side is finite by  \eqref{ell} and the above. Hence, by positivity of the transition jumps we obtain
\begin{align*}
\E[| \Delta_{j,k} |]\leq \E[| \Delta_{k,j} +  \Delta_{k,j}  |]<\infty\qquad \text{and}\qquad \E[| \Delta_{k,j} |]\leq \E[ |\Delta_{k,j} +  \Delta_{k,j}  |]<\infty.
\end{align*}
In total we proved that $ \Delta_{j,k}$, $\Delta_{k,j}$, $\xi^j_1$ and $\xi^k_1$ all have finite absolute mean which confirms (iii).\\ \smallskip

\underline{(iv) $\Leftrightarrow$ (i) $\Leftrightarrow$ (ii) $\Leftrightarrow$ (iii):}
First note that under $P^{0,k}$, $\sigma_1$ has finite first moment so that 
$$
\lim_{n\to\infty} \frac{\sigma_n-\sigma_0}n=\lim_{n\to\infty} \frac{\sum_{i=1}^n\sigma_i}{n}=E^{0,k}[\sigma_1].
$$
Assume that the limit $\lim_{t\to\infty} \xi_t/t$ exists almost surely. In this case the limit is equal to 
$$
\lim_{n\to\infty} \frac 1{n E^{0,k}[\sigma_1]} \sum_{l=1}^n(\xi_{\sigma_l}-\xi_{\sigma_{l-1}}).
$$
However, considering the case of strong laws of large numbers for random walks (cf. Theorem 7.2 of \cite{kyprianou}), the latter limit exists and is finite if and only if $E^{0,k}[|\xi_{\sigma_1}|]<\infty$, in which case the limit above must equal  $E^{0,k}[\xi_{\sigma_1}]/ E^{0,k}[\sigma_1]$.  It follows that (iv) implies (i)-(iii) and also that $\lim_{t\to\infty}\xi_t/t$ has the second claimed limit in \eqref{nn}.\\
  
  Conversely,  now assuming the equivalent statements (i), (ii) and (iii), in particular (ii), we can   conclude that
\begin{align}\label{mm}
\lim_{n\to\infty} \frac {\xi_{\sigma_n}}{\sigma_n} =E^{0,k}[\xi_{\sigma_1}]
\end{align}
almost surely by the strong law of large numbers for random walks. Next, we need that

  \begin{align}\label{lll}
  E^{0,k}\Big[\sup_{t\in[0,\sigma_1]} |\xi_t|\Big]<\infty
  \end{align}
which can be seen as follows: By the triangle inequality and (\ref{tired}), 
\[
\sup_{t\in[0,\sigma_1]} |\xi_t|\leq \sum_{i\in E}\sup_{t\in[0,\sigma_1]}|\xi^i_t| + \sum_{i\neq j}\sum_{\ell= 1}^{n_{i,j}(\sigma_1)}|\Delta_{i,j}^\ell|.
\] 
 The expectation of the right hand side is finite thanks to the independence of $\xi^i$, $i\in E$ and $J$, the assumption (iii) and the remark at the very beginning of this proof. 
Now we use \eqref{lll} to deduce
$$
\lim_{n\to\infty} \frac{\sup_{t\in[\sigma_{n-1},\sigma_n]}|\xi_t-\xi_{\sigma_{n-1}}|}{n}=0, 
$$
which then implies in combination with \eqref{mm} almost sure convergence of $\xi_t/t$ to a finite constant which is (iv).\\

It remains to verify \eqref{nn} under any of the equivalent conditions (i) to (iv). The first equality is the law of large numbers \eqref{mmm} under finite mean and the second equality was already derived in the argument for (iv) implies (i)-(iii).
\end{proof}

A.10. {\bf Tightness of the overshoots.}
We now characterise when a general MAP has tight overshoots. That it to say, taking account of the conclusion in Theorem \ref{MRT}, we provide necessary and sufficient conditions for $E^{0,\pi}[H^+_1]<\infty$, thereby giving a proof of Theorem \ref{over}. 

\bigskip

{\it Although we have assumed the non-lattice condition in this Appendix, the results given below do not need  it. }

\bigskip

\begin{theorem}\label{theo1509-1}
The MAP $(\xi,J)$ has tight overshoots if and only if $\xi_1$ has finite absolute moment and
\begin{itemize}\item[(i)]  $(\xi, J)$ drifts to $+\infty$; or
\item[(ii)] $(\xi,J)$ oscillates  and satisfies
 \begin{itemize}
 \item[\mbox{ }] $$\hspace{-4.5cm}\text{\rm\bf (TO)} \hspace{3cm}
 \int_{\kappa}^\infty \frac {x \,\Pi([x,\infty))}{1+\int_{0}^x \int_y^\infty \Pi((-\infty,-z])\, dz\, dy}\, dx<\infty$$
 \end{itemize}
  for one (any) $\kappa>0$ and 
  \begin{align}\label{ii}
  	\Pi:=\sum_{i\not=j, i,j\in E} q_{i,j} \mathcal L(\Delta_{i,j}) + \sum_{i\in E} \Pi_i,
  \end{align}	
  	  where $\Pi_i$ is the L\'evy measure of the $i$-th L\'evy process and $\mathcal L(\Delta_{i,j})$ is the probability distribution of the transition jump from $i$ to $j$ in Proposition \ref{p1}.

 \end{itemize}

\end{theorem}

In order to prove the theorem it suffices to analyze tightness of the overshoots on the discrete   time skeleton embedded in $(\xi, J)$ at the return times of the Markov chain to a fixed state $k\in E$. As in A.9 we let $\sigma_0=\inf\{t\geq 0\colon J_t=k\}$ and inductively define, for $n\in\N$,
\begin{align*}
\sigma_{n+1}=\inf\{t> \sigma_{n} \colon J_t=k \text{ and }\exists s\in(\sigma_{n-1},t)\text{ with } J_s\not=k\}
\end{align*}
so that $(\xi_{\sigma_n})_{n\in\N_0}$ is a Markov chain.
\begin{lemma}\label{le0809-1}
The MAP $(\xi, J)$ has tight overshoots if and only if the Markov chain  $(\xi_{\sigma_n})_{n\in\N_0}$ has tight overshoots under~$P^{0,k}$.
\end{lemma}
\begin{proof}
For $x,s\geq 0$  consider the stopping times 
$$
\rho_x =\inf\{t: \xi_t\geq x, J_t=k, J_{t-}\not=k\} \text{ \ and \ } \sigma(s)=\inf\{t> s: J_t=k, J_{t-}\not=k\}.
$$
For $c\geq 0$ one has
$$
\{\xi_{\rho_x}-x\geq 3c\}\subset \{\xi_{\tau_{x+c}^+}-(x+c)\geq c\} \cup\Bigl\{ \sup_{s\in [\tau_{x+c}^+, \sigma(\tau_{x+c}^+)]} |\xi_s-\xi_{\tau_{x+c}^+}|\geq c\Bigr\}.
$$
Indeed, in the case where the overshoot of the discrete time process $(\xi_{\sigma_n})$ is larger than $3c$ and the overshoot of the continuous time process over $x+c$ is smaller  than $c$, one has $\xi_{\tau_{x+c}^+}\in[x+c,x+2c]$ so that the process has to oscillate between time $\tau_{x+c}^+$ and the next entry of $J$ into $k$ at time   $\sigma(\tau_{x+c}^+)$ by at least $c$. 
For every $i,j\in E$ and $x\geq 0$ one has 
$$ P^{0,i}\Bigl(\sup_{s\in [\tau_{x+c}^+, \sigma(\tau_{x+c}^+)]} |\xi_s-\xi_{\tau_{x+c}^+}|\geq c \,\Big |\,J_{\tau_{x+c}^+}=j\Bigr)= P^{0,j}\Bigl(\sup_{s\in [0,\sigma_1]} |\xi_s|\geq c\Bigr)
$$
and since finite families and mixtures thereof are always tight, there exists a decreasing function $g_2:[0,\infty)\to[0,1]$ with limit $0$ such that
$$
P^{0,i}\Bigl(\sup_{s\in [\tau_{x+c}^+, \sigma(\tau_{x+c}^+)]} |\xi_s-\xi_{\tau_{x+c}^+}|\geq c \Bigr)\leq g_2(c),\text{ \ \ for } i\in E,x,c\geq 0.
$$
If the continuous time process has tight overshoots, then there is a function $g_1:[0,\infty)\to[0,1]$ with limit $0$ such that
$$
P^{0,i}(\xi_{\tau_{x+c}^+}-(x+c)\geq c)\leq g_1(c), \text{ \ \ for } i\in E,x,c\geq 0,
$$
so that altogether
$$
\sup_{i\in E,x\geq 0} P^{0,i}(\xi_{\rho_x}-x\geq 3c)\leq g_1(c)+g_2(c)
$$
and the overshoots of the discrete time process are tight.  \smallskip

  The converse direction follows analogously. Using that
$$\{\xi_{\tau_x^+}-x\geq 2c\} \subset \{\xi_{\rho_x}-x\geq c\} \cup \Bigl\{ \sup_{s\in [\tau_{x}^+, \sigma(\tau_{x}^+)]} |\xi_s-\xi_{\tau_{x+c}^+}|\geq c\Bigr\}
$$
one deduces that the continuous time process has tight overshoots if the discrete time process has tight overshoots under any of the laws $P^{0,i}$. Further using that for $i\in E$ and $x,c\geq 0$ one has
$$
P^{0,i}( \xi_{\rho_x}-x\geq c) \leq P^{0,i}\Bigl( \sup_{s\in [0,\sigma(0)]} \xi_s\geq c\Bigr) + E^{0,i}\Bigl[ P^{\xi_{\sigma(0)}\wedge x,k} (\xi_{\rho_x}-x\geq c)\Bigr]
$$
one deduces that tightness of the overshoots of the discrete time process under the law $P^{0,k}$ induces tightness under any law $P^{0,i}$ with $i\in E$.
\end{proof}

The following lemma is a consequence of Theorem 8 of \cite{DonMal}.

\begin{lemma} \label{le2210-1}
A random walk has tight overshoots if and only if the distribution of its increments has finite absolute moment, it  drifts to infinity or oscillates and the distribution $\Pi$ of its increments satisfies the integrabililty condition (TO).
\end{lemma}

The next result will be helpful later to separate big jumps from small jumps in the L\'evy processes corresoponding through Proposition \ref{p1} to the MAP $(\xi,J)$.

\begin{lemma}\label{le0809-2}
Let $X,Y$ be real random variables with $Y$ being square integrable, then the distribution of $X$ satisfies (TO) if and only if the distribution of $X+Y$ satifies (TO).
\end{lemma}

\begin{proof}
It suffices to show that $X+Y$ satisfies (TO), if $X$ satisfies (TO). For the reverse statement the same argument applies with the use of $-Y$ instead of $Y$. We use that for $z\geq0$
$$
\P(X+Y\geq z)\leq \P(X\geq z/2)+\P(Y\geq z/2) \text{ \ and \ } \P(X+Y\leq -z) \geq \P(X\leq -2z)- \P(Y\geq z)
$$
to deduce that
\begin{align*}&\quad\int_{\kappa}^\infty \frac {x \,\P(X+Y\geq x)}{1+\int_{0}^x \int_y^\infty \P(X+Y\leq -z)\, dz\, dy}\, dx\\
&\leq\int_{\kappa}^\infty \frac {x \,\P(X\geq x/2)}{1+\int_{0}^x \int_y^\infty (\P(X\leq -2z)- \P(Y\geq z))_+ \, dz\, dy}\, dx+ \int_{\kappa}^\infty x \,\P(Y\geq x/2)\, dx.
\end{align*}
The latter intergral is finite since $Y$ has finite second moment and the proof is finished once we showed that the former integral is finite. One has
$$
\int_{0}^\infty \int_y^\infty \P(Y\geq z) \, dz\, dy  =\frac12 \E[Y_+^2]<\infty
$$
and taking $c\geq 1$ with $c>\E[Y_+^2]$ we conclude with substitution that
\begin{align*}
&\quad\int_{\kappa}^\infty \frac {x \,\P(X\geq x/2)}{1+\int_{0}^x \int_y^\infty (\P(X\leq -2z)- \P(Y\geq z))_+ \, dz\, dy}\, dx\\
&\leq 4c \int_{\kappa/2}^\infty \frac {x \,\P(X\geq x)}{c+\int_{0}^{2x} \int_y^\infty (\P(X\leq -2z)- \P(Y\geq z))_+ \, dz\, dy}\, dx\\
&\leq 4c \int_{\kappa/2}^\infty \frac {x \,\P(X\geq x)}{c/2+\int_{0}^{2x} \int_y^\infty \P(X\leq -2z) \, dz\, dy}\, dx\\
&= 4c \int_{\kappa/2}^\infty \frac {x \,\P(X\geq x)}{c/2+\frac 14 \int_{0}^{4x} \int_y^\infty \P(X\leq -z) \, dz\, dy}\, dx\\
&\leq c' \int_{\kappa/2}^\infty \frac {x \,\P(X\geq x)}{1+ \int_{0}^{x} \int_y^\infty \P(X\leq -z) \, dz\, dy}\, dx<\infty,
\end{align*}
where $c'= \max\{8,16c\}$. 
\end{proof}

\begin{lemma}\label{le0809-3}
 Let $\Pi_i, i\in E,$ be probability distributions on $\R$ and let $\{X^{i,n}\colon i\in E, n\in \N\}$ be a family of independent random variables with $X^{i,n}\sim \Pi_i$. Define
$$
Z=\sum_{n=1}^N X^{Y_n,n}$$
with $(Y_n)_{n\in\N}$ being an $E$-valued process and $N$ an $\N_0$-valued random variable both being jointly independent of $(X^{i,n})$.
 If furthermore we suppose $\E[N^3] <\infty$ and $\P(i\in\{Y_1,\dots,Y_N\})>0$ for $i\in E$, then the following properties are equivalent: 
\begin{enumerate}
\item[(i)] The distribution of $Z$ satisfies (TO).
\item[(ii)] For one (any) sequence $(\rho_i)_{i\in E}$ of strictly positive numbers
\begin{align*}
	\Pi^\mathrm{sum}(\cdot):=\sum_{i\in E} \rho_i\,\Pi_i(\cdot)
\end{align*}
satisfies (TO).
\item[(iii)] The measure $\Pi^\mathrm{max}$ on $\R\backslash\{0\}$ defined by 
\begin{align*}
	\Pi^\mathrm{max}([t,\infty))&=\max_{i\in E} \Pi_i([t,\infty))\\
	\Pi^\mathrm{max}((-\infty,-t])&=\max_{i\in E} \Pi_i((-\infty,-t])
\end{align*}
for $t>0$ satisfies~(TO).
\end{enumerate}
\end{lemma}

\begin{proof}
The equivalence of (ii) and (iii) follows immediately from the definition of (TO), the estimate 
$$
\min_{i\in E} \rho_i \, \Pi^\mathrm{max}([x,\infty))\leq \Pi^\mathrm{sum}([x,\infty))\leq \sum_{i\in E} \rho_i\, \Pi^\mathrm{max}([x,\infty)),\text{ \ \ for }x\geq0,
$$
and its analogues version for the set $(-\infty,-x]$. It remains to show that property (i) is equivalent to properties (ii) and (iii).
\smallskip

  We start with proving that (iii) implies (i). Note that, for $x\geq 0$,
\begin{align}\label{eq1509-1}
	\P(Z\geq x|N) \leq N\,\Pi^{\mathrm{max}}([x/N,\infty)).
\end{align}
Furthermore, for any $i\in E$ there exists $1\leq n_i'\leq n_i$ such that
$$
\P(Y_{n_i'}= i, N= n_i) >0 .
$$
Hence, for all $\kappa_i\in [0,\infty)$, one finds
$$
\P(Z\leq -z)\geq \P(Y_{n_i'}= i, N= n_i) \, \P\Bigl(\sum_{n=1}^{n_i} \1_{\{n\not= n_i'\}} X^{Y_n,n} \leq \kappa_i\,\Big|\,Y_{n_i'}= i, N= n_i\Bigr)\, \P(X^{i,1}\leq -z-\kappa_i).
$$
Now we fix $\kappa_i$ such that $$q_i:= \P(Y_{n_i'}= i, N= n_i) \, \P\Bigl(\sum_{n=1}^{n_i} \1_{\{n\not= n_i'\}} X^{Y_n,n} \leq \kappa_i\,\Big|\,Y_{n_i'}= i, N= n_i\Bigr)>0.$$
We set $\kappa=\max\{\kappa_i:i\in E\}$ and $q=\min\{q_i:i\in E\}$ and get, for $z\geq 0$,\begin{align}\label{eq1509-2}
\P(Z\leq -z)\geq q \, \max_{i\in E} \P(X^{i,1}\leq -z-\kappa) = q\, \Pi^{\mathrm{max}}((-\infty,-z-\kappa]).
\end{align}
Combining this  estimate with~(\ref{eq1509-1}) we get that
\begin{align*}
&\quad \int_{[\kappa,\infty)} \frac {x \,\P(Z\geq x)}{1+ \int_{0}^x\int_y^\infty \P(Z\leq -z)\, dz\, dy}\, dx\\
&\leq  \int_{[\kappa,\infty)} \frac {x \, \E[N \Pi^{\mathrm{max}}([x/N,\infty)]}{1+q\int_{\kappa}^x \int_y^\infty \Pi^{\mathrm{max}}((-\infty,-2z])\, dz\, dy}\, dx.
\end{align*}
Since $\int_0^\kappa\int_y^\infty \P(Z\leq -z)\,dz\,dy$ is finite we conclude that there exists a constant  $c>0$ such that for $x\geq \kappa$
$$
1+ q\int_{\kappa}^x \int_y^\infty \Pi^{\mathrm{max}}((-\infty,-2z])\, dz\, dy\geq c\Bigl(1+ \int_{0}^x \int_y^\infty \Pi^{\mathrm{max}}((-\infty,-2z])\, dz\, dy\Bigr).
$$
Hence, we have
\begin{align*}
&\quad \int_{[\kappa,\infty)} \frac {x \,\P(Z\geq x)}{1+\int_{0}^x \int_y^\infty \P(Z\leq -z)\, dz\, dy}\, dx\\
&\leq c^{-1} \sum_{n=1}^\infty \P(N=n) n \int_{(0,\infty)} \frac {x \, \Pi^{\mathrm{max}}([x/n,\infty)]}{1+\int_{0}^x \int_y^\infty \Pi^{\mathrm{max}}((-\infty,-2z])\, dz\, dy}\, dx\\
&= c^{-1} \sum_{n=1}^\infty \P(N=n) n^3 \int_{(0,\infty)} \frac {x \, \Pi^{\mathrm{max}}([x,\infty)]}{1+4^{{-1}}\int_{0}^{2nx} \int_y^\infty \Pi^{\mathrm{max}}((-\infty,-z])\, dz\, dy}\, dx\\
&\leq  4c^{-1} \E[N^3] \, \int_{(0,\infty)} \frac {x \, \Pi^{\mathrm{max}}([x,\infty)]}{1+\int_{0}^{x} \int_y^\infty \Pi^{\mathrm{max}}((-\infty,-z])\, dz\, dy}\, dx.
\end{align*}

Next, we consider the converse direction. In analogy to the derivation of~(\ref{eq1509-2}), one sees that there are constants $\kappa,q>0$ such that
$$
\P(Z\geq z)\geq q \, \Pi^{\mathrm{max}}([z+\kappa,\infty)),
$$
for all $z\geq 0$.
Further, it is also the case that
\begin{align*}
\int_0^x\int_y^\infty \P(Z\leq -z)\,dz\,dy&\leq \sum_{n=1}^\infty \P(N=n) \int_0^x\int_y^\infty n\, \Pi^{\mathrm{max}}((-\infty,-z/n])\,dz\,dy\\
&=  \sum_{n=1}^\infty \P(N=n) n^3\int_0^{x/n}\int_y^\infty  \Pi^{\mathrm{max}}((-\infty,-z])\,dz\,dy\\
&\leq \E[N^3]\, \int_0^{x}\int_y^\infty  \Pi^{\mathrm{max}}((-\infty,-z])\,dz\,dy,
\end{align*}
so that we have
\begin{align*}
&\quad\int_{[2\kappa,\infty)} \frac {x \,\Pi^{\mathrm{max}}([x,\infty))}{1+\int_{0}^x \int_y^\infty \Pi^{\mathrm{max}}((-\infty,-z])\, dz\, dy}\, dx\\
&\leq q^{-1}\, (\E[N^3]\vee1) \int_{[2\kappa,\infty)} \frac {x \,\P(Z\geq x/2)}{1+ \int_{0}^x \int_y^\infty \P(Z\leq -z)\, dz\, dy}\, dx\\
&= 4q^{-1}\, (\E[N^3]\vee 1)\int_{[\kappa,\infty)} \frac {x \,\P(Z\geq x)}{1+\int_{0}^x \int_y^\infty \P(Z\leq -z)\, dz\, dy}\, dx.
\end{align*}
The proof is now complete.
\end{proof}

\begin{proof}[Proof of Theorem \ref{theo1509-1}]
We again consider the process $\xi$ at the discrete set of return times to the state $k$.
By Lemma~\ref{le0809-1} tightness of the overshoots of the MAP is equivalent to tightness of the overshoots of the discrete time process $(\xi_{\sigma_n})_{n\in\N_0}$ under the law $P^{0,k}$ which is the underlying measure in the following considerations. By the Markov property and the translation invariance of the MAP, the process $(\xi_{\sigma_n})$ has iid increments  and starts in $0$ and thus is a random walk.
By Lemma~\ref{le2210-1}, $(\xi_{\sigma_n})$ has tight overshoots if and only if $\xi_{\sigma_1}$ has finite absolute moment and either
\begin{itemize}
	\item  $(\xi_{\sigma_n})$ drifts to infinity, or
	\item $(\xi_{\sigma_n})$ oscillates and the distribution of $\xi_{\sigma_1}$ satisfies (TO).
\end{itemize}
By Theorem~\ref{theo1311-1}, Formula \eqref{nn}, the latter properties are equivalent to the ones obtained when replacing the discrete time process $(\xi_{\sigma_n})$ by the continuous time process $(\xi_t)$ and keeping the (TO) property for $\xi_{\sigma_1}$.

To finish the proof it remains to show that in the oscillating case with finite absolute moment one has the equivalence \begin{align*}
	\mathcal L(\xi_{\sigma_1})\text{ satisfies (TO)}\quad\Longleftrightarrow\quad \Pi\text{ from \eqref{ii} satisfies (TO)}.
\end{align*}
In order to do so let us identify the distribution of $\xi_{\sigma_1}$. Enumerate the  times at which either $\xi$ has jumps with modulus larger than $1$ or $J$ changes its state in increasing order $0\leq \tau_1<\tau_2<\ldots$ and represent $\xi_{\sigma_1}$ as telescopic sum
\begin{align}\label{iiii}
\xi_{\sigma_1}= \sum_{j:\tau_j\leq \sigma_1} (\xi_{\tau_j}-\xi_{\tau_{j}-})+\sum_{j:\tau_j\leq \sigma_1}  (\xi_{\tau_j-}- \xi_{\tau_{j-1}})
\end{align}

with $\tau_0:=0$. Using the representation from Proposition~\ref{p1}, we can identify the conditional distributions of the terms appearing in the former sum  when conditioning on $J$ and the set of times $\{\tau_1,\tau_2,\dots\}$: If $\tau_j$ is triggered by a large jump of the L\'evy process (meaning that the process $J$ does not switch states at that time)   the conditional distribution of  $\xi_{\tau_j}-\xi_{\tau_{j}-}$ is the normalised L\'evy measure restricted to jumps larger than one of the L\'evy process that is switched on by the modulating chain. If $\tau_j$ is triggered by a change of $J$, then the conditional distribution of $\xi_{\tau_j}-\xi_{\tau_{j}-}$ is $\mathcal L(\Delta_{J(\tau_j-),J(\tau_j)})$ with $U$ as in Proposition~\ref{p1}. The random number of $j$'s with $\tau_j\leq \sigma_1$ has finite third moment and applying Lemma~\ref{le0809-3} we get that
\begin{align*}
\mathcal L\Big(\sum_{j:\tau_j\leq \sigma_1} (\xi_{\tau_j}-\xi_{\tau_{j}-})\Big)\text{ satisfies (TO)}\quad \Longleftrightarrow\quad
\sum_{i\not=j} q_{i,j} \,\mathcal L(\Delta_{i,j})+\sum_{i\in E}\Pi_i|_{B(0,1)^c}\text{ satisfies (TO)}.
\end{align*}
An elementary calculation furthermore shows that
\begin{align*}
\sum_{i\not=j} q_{i,j} \,\mathcal L(\Delta_{i,j})+\sum_{i\in E}\Pi_i|_{B(0,1)^c}\text{ satisfies (TO)}\quad \Longleftrightarrow\quad \Pi\text{ from \eqref{ii} satisfies (TO)}.
\end{align*}
Combining the two equivalences with \eqref{iiii} the theorem is proved (compare Lemma~\ref{le0809-2}) if the remainder $\sum_{j:\tau_j\leq \sigma_1} (\xi_{\tau_{j-}}-\xi_{\tau_{j-1}})$ has finite second moment.\\ 
However, the latter term is just the  value of a MAP starting in $(0,k)$ evaluated at the time of the first return of $J$ to $k$ with an appropriately modified evolution: the L\'evy measures need to be replaced by the old ones restricted to the unit ball and the process has no discontinuity when $J$  switches states. Such a MAP obviously has finite second moment.
\end{proof}

\end{document}